\documentclass[11pt,a4paper]{article}
\usepackage[colorlinks,citecolor=cyan,linkcolor=magenta]{hyperref}
\usepackage[latin1]{inputenc}
\usepackage{amsmath}
\usepackage{amsfonts}
\usepackage{amssymb}
\usepackage{theorem}
\usepackage{subcaption}
\usepackage{graphicx}
\usepackage{xspace}
\usepackage{tikz}
\usepackage{fullpage}

\newcommand{\ndg}[1]{| \kern -.25mm \|{#1}| \kern -.25mm \|}
\newcommand{\nsdg}[1]{| \kern -.25mm \|{#1}| \kern -.25mm \|_s}

\newcommand{\ltwo}[2]{\|{#1}\|_{#2}}

\newcommand{\el}{ T \in \mathcal{T} }
\newcommand{\ud}{\mathrm{d}}

\newcommand{\av}[1]{\{#1\}}
\newcommand{\norm}[2]{\|#1\|_{#2}}
\newcommand{\snorm}[2]{|#1|_{#2}}
\newcommand{\dint}{\text{\rm int}}
\newcommand{\mbf}[1]{\mbox{\boldmath$\rm{#1}$}}

\newcommand{\jump}[1]{ [#1] }

\newcommand{\rR}{\ensuremath{\mathbb R}\xspace}
\newcommand{\reals}{\rR}

\newcommand{\ddd}{{\rm D}}
\newcommand{\dn}{{\rm N}}
\newcommand{\db}{{\rm b}}

\DeclareMathOperator{\diam}{diam}

\renewcommand{\tilde}[1]{\widetilde{#1}}
\renewcommand{\hat}[1]{\widehat{#1}}
\makeatletter
\newcommand*{\rom}[1]{\text{\expandafter\@slowromancap\romannumeral #1@}}
\makeatother

\newtheorem{corollary}{Corollary}[section]
\newtheorem{lemma}[corollary]{Lemma}
\newtheorem{theorem}[corollary]{Theorem}
\newtheorem{proposition}[corollary]{Proposition}

\newtheorem{remark}[corollary]{Remark}

\newtheorem{assumption}[corollary]{Assumption}

\newcommand{\qed}{ \vspace{-0.5cm} \hfill $\Box$ }
\newenvironment{proof}[1][Proof.]{\begin{trivlist}
\item[\hskip \labelsep {\bfseries #1}]}{\end{trivlist}\qed}

\begin{document}
\title{Recovered finite element methods \\ on polygonal and polyhedral meshes}
\author{
Zhaonan Dong\thanks{
Department of Mathematics,
University of Leicester,
University Road,
Leicester LE1 7RH,
UK
{\tt{zd14@le.ac.uk}}.
}
\and Emmanuil H.~Georgoulis\thanks{
Department of Mathematics,
University of Leicester,
University Road,
Leicester LE1 7RH,
UK and School of Applied Mathematical and Physical Sciences, National Technical University of Athens, Zografou 15780, Greece
{\tt{Emmanuil.Georgoulis@le.ac.uk}}.
}\and Tristan Pryer\thanks{
Department of Mathematics and Statistics,
University of Reading,
Whiteknights,
PO Box 220,
Reading RG6 6AX,
UK
{\tt{T.Pryer@reading.ac.uk}}.
}}
\date{\today}

\maketitle

\begin{abstract}
\noindent
\emph{Recovered finite element methods (R-FEM)} has been recently introduced in \cite{RFEM} for meshes consisting of simplicial and/or box-type meshes. Here, utilising the flexibility of R-FEM framework, we extend their definition on polygonal and polyhedral meshes in two and three spatial dimensions, respectively. A key attractive feature of this framework is its ability to produce \emph{conforming} discretizations, yet involving \emph{only} as many degrees of freedom as discontinuous Galerkin methods over general polygonal/polyhedral meshes with potentially many faces per element. A priori error bounds are shown for general linear, possibly degenerate, second order advection-diffusion-reaction boundary value problems. A series of numerical experiments highlights the good practical performance of the proposed numerical framework. 
\end{abstract}

\section{Introduction}

Recently, there has been a considerable interest in the construction of Galerkin-type numerical methods over meshes consisting of general polygons in two dimensions or general polyhedra in three dimensions, henceforth termed collectively as \emph{polytopic}, as opposed to the classical Galerkin methods employing simplicial and/or box-type meshes. This interest is predominantly motivated by the potential reduction in the total numerical degrees of freedom required for the numerical solution of PDE problems. This is particularly pertinent in the context of adaptive computations for evolution PDE problems, where dynamic mesh modification is widely accepted as a promising tool for the reduction of computational complexity in both Eulerian and Lagrangian contexts. Galerkin procedures over polytopic meshes have been also proposed in the context of interface problems (porosity profiles, interfaces, etc.), as well as in the context of coarse correction computations in multilevel solvers for elliptic boundary-value problems.

Popular polytopic methods include the virtual element method \cite{VEM6,UnifiedVEM}, which is itself an evolution of the so-called mimetic finite difference methods \cite{MimeticBook2014}, polygonal finite element methods \cite{Tabarraei:Sukumar:2004}, {composite finite element methods \cite{hackbusch_sauter_cfe_nm,MR2439507} }and various discontinuous Galerkin (dG) approaches, ranging from one-field interior penalty dG methods \cite{DiPietroErn,DGpoly1,DGpoly2,DGpolybook}, to hybridized formulations \cite{HDG,HHOandHDG}. DG methods are attractive as one can control the number the global numerical degrees of freedom independently of the mesh topology (i.e., the connectivity of the nodes/faces/elements), whereas polygonal finite elements and virtual element methods are attractive as they involve conforming approximation spaces. To construct such conforming spaces, polygonal finite element and/or virtual element methods involve basis functions which are dependent on the mesh topology: roughly speaking, even the lowest order spaces require as many basis functions as the number of mesh nodes, thereby hindering a potential substantial complexity reduction by the use of polytopic elements with ``many'' faces. 

In this work, motivated by the recent recovered finite element framework presented in \cite{RFEM}, we construct \emph{conforming} methods over polytopic meshes whose set of degrees of freedom is \emph{independent} of the number of vertices/edges/faces of each element. The proposed method depends, instead, on the choice of a sub-triangulation of the polytopic meshes. Crucially, however, the computational complexity of the method is \emph{independent} of the cardinality of the simplices in the sub-triangulation. More specifically, the recovered finite element method (R-FEM) on polytopic meshes combines completely discontinuous local \emph{polynomial} spaces, resulting, nonetheless, to conforming approximations. 

To fix ideas, let us consider an elliptic boundary value problem with homogeneous Dirichlet boundary conditions. Let $\mathcal{E}:V_h\to \tilde{V}_h\cap H^1_0(\Omega)$ an operator mapping a \emph{discontinuous} element-wise polynomial space $V_h$ over a polytopic mesh onto a space of \emph{continuous} piecewise polynomial space $\tilde{V}_h\cap H^1_0(\Omega)$ over a, generally speaking, finer simplicial mesh arising from a sub-triangulation of the polytopic mesh; such \emph{recovery} operators $\mathcal{E}$ can be constructed locally, e.g., by (weighted) averaging of the nodal degrees of freedom \cite{KP,MR0400739}. We can now consider the method: find $u_h\in V_h$, such that
\[
\int_{\Omega}\nabla \mathcal{E}(u_h)\cdot \nabla \mathcal{E}(v_h)\ud x + s(u_h,v_h) = \int_{\Omega}f\mathcal{E}(v_h)\ud x, \quad\text{ for all } v_h\in V_h,
\]
for $f\in H^{-1}(\Omega)$ and a suitable \emph{stabilization} $s(\cdot,\cdot):V_h\times V_h\to\mathbb{R}$, whose functionality is the treatment of the kernel $\{0\neq v_h\in V_h: \mathcal{E}(v_h)=0\}$ to achieve unisolvence. Crucially, despite using element-wise discontinuous polynomial trial and test space $V_h$, the method also produces simultaneously a conforming approximation $\mathcal{E}(u_h)$. We note that the above method yields, in general, different numerical solutions to those one would get by postprocessing standard dG approximations on polytopic meshes via the recovery operator $\mathcal{E}$. In the limit case of the above R-FEM posed on a simplicial mesh (rather than a general polytopic one), $\mathcal{E}(u_h)$ corresponds to the classical conforming FEM approximations for certain choices of $\mathcal{E}$ \cite{RFEM}. Therefore, in this sense, R-FEM is an extension of classical finite element methods to polytopic meshes. An interesting property of the proposed method is that the user has access to the computed approximate solution at every point in the computational domain. This may be of practical interest both in the context of further post-processing and in the visualisation of the computation on standard widely available software.

The above example of an elliptic problem is intended to highlight some of the attractive features for R-FEM: a) conformity is not hard-wired in the approximation spaces and b) there is considerable flexibility in the particular choice of the recovery operator $\mathcal{E}$, the finite element space $\mathcal{E}(V_h)$, and the stabilisation $s$ used. Moreover, since $\{0\neq v_h\in V_h: \mathcal{E}(v_h)=0\}$ is allowed to be \emph{non-trivial} by construction, R-FEM offers significant flexibility in combining various types of numerical degrees of freedom (nodal, modal, moments, etc.) for different elements in the same mesh. This may be of interest in the treatment of interface problems.

{The remainder of this work is structured as follows. In Section \ref{model}, we introduce the problem and define a set of polytopic meshes. In Section \ref{sec:rfem_FES}, we introduce the FEM spaces and the recovery operators. Section \ref{concept} presents the concepts and ideas for designing R-FEM. In Section \ref{sec:alternative}, we define the R-FEM for the model problem. The a priori error analysis for R-FEM is presented in Section \ref{sec:apriori}. Finally, the practical performance of the proposed R-FEM is tested through a series of numerical examples in Section \ref{numerics}.}

\section{Model problem}\label{model}

Throughout this work we denote the standard Lebesgue spaces by
$L^p(\omega)$, $1\le p\le \infty$, $\omega\subset\mathbb{R}^d$,
$d=2,3$, with corresponding norms $\|\cdot\|_{L^p(\omega)}$; the norm
of $L^2(\omega)$ will be denoted by $\ltwo{\cdot}{\omega}$ for
brevity. Let also $W^{s,p}(\omega)$ and
$H^s(\omega):=W^{s,2}(\Omega)$, be the Banach and Hilbertian Sobolev
space of index $s\in\mathbb{R}$ of real-valued functions defined on
$\omega\subset\mathbb{R}^d$, respectively, constructed via standard
interpolation and/or duality procedures for $s\notin
\mathbb{N}_{0}$. For $H^s(\omega)$, we denote the corresponding norm
and seminorm by $\norm{\cdot}{s,\omega}$ and
$\snorm{\cdot}{s,\omega}$, respectively. We also denote by
$H^1_0(\omega)$ the space of functions in $H^1(\omega)$ with vanishing
trace on $\partial \omega$.

For $\Omega$ a bounded open polygonal domain in $\mathbb{R}^d$, $d\in
\mathbb{N}$, with $\partial\Omega$ denoting its boundary, we consider
advection-diffusion-reaction problem
\begin{equation}\label{pde}
-\nabla\cdot a\nabla u + \mbf{b}\cdot \nabla u + cu=f\quad\text{in } \Omega,
\end{equation}
where $f\in L^2(\Omega)$, $\mbf{b}\in [W^{1,\infty}(\Omega)]^d$, $c\in
L^\infty(\Omega)$, for some definite diffusion tensor $a\in
[L^{\infty}(\Omega)]^{d\times d}$ satisfying
\begin{equation}\label{nonneg}
\zeta^T a(x)\zeta \ge 0, \quad \text{for all } \zeta\in \mathbb{R}^d,
\end{equation}
for almost every $x\in\bar{\Omega}$. This class of problem is often
termed \emph{PDEs with non-negative characteristic form} \cite{OR} and
includes elliptic, parabolic, first order hyperbolic as well as other
non-standard types of PDEs, such as ultra-parabolic and various
classes of linear degenerate equations. In particular, the important
family of linear Kolmogorov-Fokker-Planck equations are of the form
\eqref{pde}.

To prescribe suitable boundary conditions, we begin by splitting
$\partial\Omega$ into
\[
\Gamma_0:= \{x\in\partial\Omega: \mbf{n}^T(x) a(x)\mbf{n}(x) > 0\},
\]
and 
\[
\Gamma_1:= \{x\in\partial\Omega: \mbf{n}^T(x) a(x)\mbf{n}(x) = 0\},
\]
with $\mbf{n}(x)$ denoting the unit outward normal vector to $\Omega$
at $x\in\partial\Omega$; the latter is further subdivided into inflow
\[
\Gamma_- :=\{x\in \Gamma_1: \mbf{b}(x)\cdot \mbf{n}(x)<0\},
\]
and outflow $\Gamma_+:= \Gamma_0 \backslash \Gamma_-$ parts of the
boundary. The ``elliptic'' part of the boundary $\Gamma_0$, is
subdivided into $\Gamma_{\ddd}$ and $\Gamma_{\dn}$, on which we can
prescribe Dirichlet and and Neumann boundary conditions,
respectively. For simplicity, we assume that $|\Gamma_{\ddd}|>0$, with
$|\cdot |$ denoting the Hausdorff measure with the dimension of its
argument, and that $\mbf{b}(x)\cdot \mbf{n}(x)\ge 0$ for almost all
$x\in\Gamma_{\dn}$. To complete the problem, we impose the boundary
conditions:
\begin{equation}\label{bcs}
\begin{aligned}
u =&\ g_{\ddd}^{},\qquad \text{on } \Gamma_{\ddd}\cup\Gamma_-,\\
a\nabla u\cdot \mbf{n} = &\ g_{\dn}^{}, \qquad \text{on } \Gamma_{\dn},
\end{aligned}
\end{equation}
for some known $g_{\ddd}^{}\in L^2(\partial\Omega)$ and $g_{\dn}^{}\in H^{1/2}(\partial\Omega)$. For convenience, we also define the set
\[
\Gamma_{\ddd}^-:=\{ x\in\partial\Omega: \mbf{b}(x)\cdot \mbf{n}(x)<0\},
\]
i.e., the inflow part of the boundary including also, possibly, points
of $\Gamma_{\ddd}$. {Similarly, we define $\Gamma_{\dn}^+ = \partial
\Omega \backslash \Gamma_{\ddd}^-$}. Additionally, we assume that
the following positivity hypothesis holds: there exists a positive
constant $\gamma_{0}$ such that
\begin{equation}\label{assumption-cb}
c_0(x):=c(x)-\frac{1}{2}\nabla\cdot \bold{b}(x)\geq \gamma_0 \quad \text{ a.e. } x\in\Omega.
\end{equation} 
The well-posedness of the boundary value problem \eqref{pde}, \eqref{bcs} has been studied in \cite{HSS_hyp}.

\section{Finite element spaces and recovery operators}
\label{sec:rfem_FES}
Let $\mathcal{T}$ be a subdivision of $\Omega$ into disjoint polygonal
elements for $d=2$ or to disjoint polyhedral elements for $d=3$;
henceforth, these will be collectively referred to as \emph{polytopic
elements}. For simplicity, we assume that the subdivision
$\mathcal{T}$ can be further subdivided into a regular (i.e., no
hanging nodes) and shape-regular simplicial triangulation
$\tilde{\mathcal{T}}$ (see, e.g., p.124 in \cite{ciarlet}), that
$\bar{\Omega}=\cup_{\el}\bar{T}$. Such a setting can be constructed,
e.g., by agglomerating simplicial elements into polytopic ones.

\begin{figure}
\centering
\subcaptionbox{A polygonal mesh $\mathcal{T}_1$.}{
\includegraphics[scale=0.35]{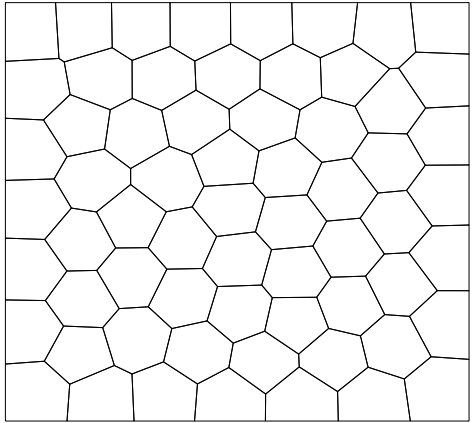}
}
\hspace{1cm}
\subcaptionbox{A simplicial subdivision $\tilde{\mathcal{T}_1}$.}{
\includegraphics[scale=0.35]{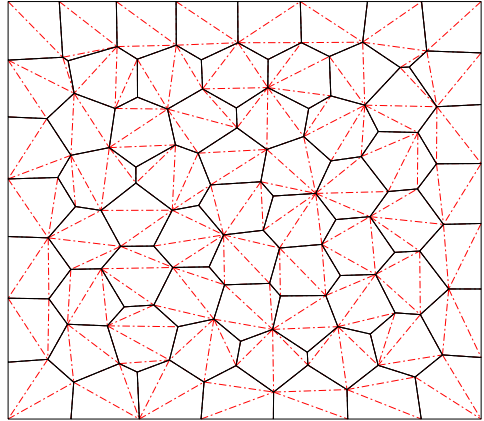}
}
\subcaptionbox{ An
example of an agglomerated mesh with many tiny faces
$\mathcal{T}_2$.}{
\includegraphics[scale=0.35]{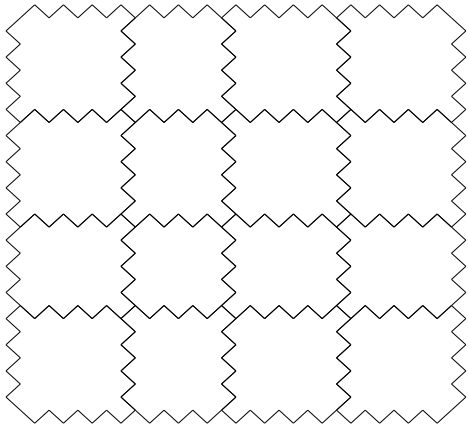}
}
\hspace{1cm}
\subcaptionbox{A simplicial background mesh $\tilde{\mathcal{T}_2}$.}{
\includegraphics[scale=0.35]{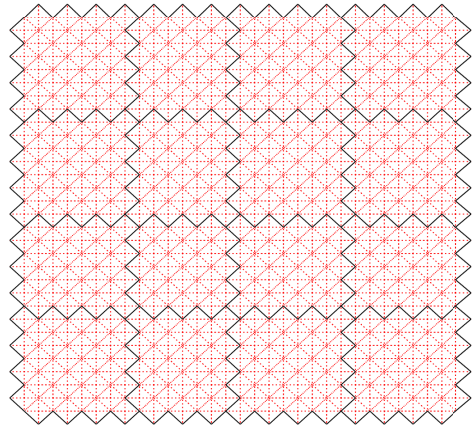}
}
\caption{Two examples of a polygonal meshes
and respective simplicial subdivisions.
}
\label{fig_one}
\end{figure}

By $\Gamma$ we shall denote the union of all ($d-1$)-dimensional faces associated with the subdivision $\mathcal{T}$ including the boundary. Further, we set $\Gamma_{\dint}:=\Gamma\backslash\partial\Omega$. Correspondingly, we define $\tilde{\Gamma}$ and $\tilde{\Gamma}_{\dint}$ for $\tilde{\mathcal{T}}$. Note that, by construction, $\Gamma\subset \tilde{\Gamma}$ and $\Gamma_{\dint}\subset \tilde{\Gamma}_{\dint}$.

For a nonnegative integer $r$, we denote the set of all polynomials of
total degree at most $r$ by $\mathcal{P}_r(T)$.
For $r \geq 1$, we consider the finite element space
\begin{equation}
V_h^r :=\{v\in L^2(\Omega):v|_{T}
\in\mathcal{P}_{r}(T),\,\el\}.
\end{equation}
We stress that $V_h^r$ is element-wise discontinuous polynomial with respect to the \emph{polytopic} mesh $\mathcal{T}$; in this context, the dimension of $V_h^r$ coincides with the dimension of discontinuous Galerkin finite element spaces on polytopic meshes, cf., \cite{DGpoly1,DGpoly2,DGpolyparabolic}. In particular, the dimension of $V_h^r$ is \emph{not} dependent on the number of vertices of the mesh $\mathcal{T}$. Correspondingly, we define 
\[
\tilde{V}_h^r :=\{v\in L^2(\Omega):v|_{T}\in\mathcal{P}_{r}(T),\,T\in\tilde{\mathcal{T}}\},
\] 
the respective discontinuous polynomial space on the sub-triangulation $\tilde{\mathcal{T}}$. Note that $V_h^r\subset \tilde{V}_h^r$.

Further, let $T^+,T^-\in\mathcal{T}$ be two (generic) elements sharing a facet
$e:=\partial T^+\cap\partial T^-\subset\Gamma_{\dint}$ with respective outward normal unit vectors $\mbf{n}^+$ and $\mbf{n}^-$ on $e$. For a function $v:\Omega\to\mathbb{R}$ that may be discontinuous across $\Gamma_{\dint}$,
we set $v^+:=v|_{e\subset\partial T^+}$, $v^-:=v|_{e\subset\partial T^-}$, and we define the jump and average by
\[ \jump{v}:=v^+\mbf{n}^++v^-\mbf{n}^-\quad\text{and}\quad \av{v}:=\frac{1}{2}(v^++v^-);\]
if $e\in \partial T\cap\partial\Omega$, we set $\jump{v}:=v^+\mbf{n}$. Also, we define $h_{T}:=\diam(T)$ and we set ${\bf h}:\Omega\backslash\Gamma\to\mathbb{R}$,
with ${\bf h}|_{T}=h_{T}$, $\el$.
Similarly, we set $\tilde{\mbf{h}}$ for the meshsize function of $\tilde{\mathcal{T}}$. Throughout this work, we assume that the families of meshes considered in this work are locally
quasi-uniform and that there exists constant $c_{\Delta}>1$, independent of the meshsizes such that
\[
c_{\Delta}^{-1} \mbf{h}\le \tilde{\mbf{h}}\le c_{\Delta}\mbf{h},
\]
uniformly as $\mbf{h}\to 0$.
Moreover, for the restriction a function $v$ on an element $T\in\mathcal{T}$, $v|_T:T\to\mathbb{R}$, which may be discontinuous across $\partial T$, we shall use the notational convention that $v^+|_{\partial T}$ signifies the trace from within $T$ while $v^-|_{\partial T}$ signifies the trace from within $\Omega\backslash T$. Using this convention we also define the \emph{signed jump} (also known as \emph{upwind} jump in the discontinuous Galerkin literature) on each face $e$ by
\[
\lfloor v\rfloor|_e :=v^+|_e-v^-|_e;
\]
note that $ |\jump{v}| = |\lfloor v\rfloor |$, i.e., the two jump notions can differ at most up to a sign on each face.

Also, we shall denote by $\partial_-T$ and by $\partial_+T$ the \emph{inflow} and \emph{outflow} parts of the boundary of an element $T$, defined as
\[
\partial_-T :=\{x\in \partial T: \mbf{b}(x)\cdot \mbf{n}(x)<0\}\quad\text{ and }\quad \partial_+T :=\{x\in \partial T: \mbf{b}(x)\cdot \mbf{n}(x)>0\},
\] 
respectively.

For the definition of the proposed method, we shall require a \emph{recovery operator} of the form 
\begin{equation}\label{recovery}
\mathcal{E}:V_h^r\to V\cap \tilde{V}_h^r,
\end{equation}
for some non-negative integer $r$, mapping element-wise discontinuous functions into functions in the solution space for the boundary value problem $V$, for some $r\in\mathbb{N}\cup\{0\}$. When the diffusion tensor $a$ is strictly positive definite (i.e., when \eqref{nonneg} holds with strict inequality) we may take $V=H^1(\Omega)$.

Recovery operators of the form \eqref{recovery} have appeared in various settings in the theory of finite element methods, e.g., \cite{MR0400739,MR1011446,MR1248895, KP,brenner,GHV}. They are typically used to recover a ``conforming'' function from a ``non-conforming'' one, often under minimal regularity requirements.

A popular and very practical example for $\mathcal{E}$ is the nodal \emph{averaging operator} for which the following celebrated stability result was proven by Karakashian and Pascal in \cite{KP}.
\begin{lemma} \label{lemma_E1}
Let $\mathcal{T}$ a polytopic mesh and $\tilde{\mathcal{T}}$ its related sub-triangulation satisfying the above assumptions. Denoting by $\mathcal{N}$ the set of all Lagrange nodes of $\tilde{V}_h^r\cap H^1(\Omega)$, the operator $\mathcal{E}_r:V_h^r\to \tilde{V}_h^r\cap H^1(\Omega)$ is defined by:
\[
\mathcal{E}_r(v)(\nu):= \frac{1}{|\omega_{\nu}|}\sum_{T\in\omega_{\nu}} v|_{T}({\nu}),
\]
with
$
\omega_{\nu}:=\bigcup_{T\in\mathcal T: \nu\in\bar{T}}T, 
$
the set of elements sharing the node $\nu\in\mathcal{N}$ and $|\omega_{\nu}|$ their cardinality. Then, the following bound holds
\begin{equation}\label{KP_stab}
\sum_{T\in\mathcal{T}} \snorm{v - \mathcal{E}_r(v)}{\alpha,T}^2 \leq C_{|\alpha|} \norm{\mbf{h}^{1/2-\alpha}\jump{v}}{\Gamma_{\dint}}^2,
\end{equation}
with $|\alpha|\in\mathbb{N}\cup\{0\}$, $C_{|\alpha|}\equiv C_{|\alpha|}(r)>0$ a constant independent of $\mbf{h}$, $v$ and $\tilde{\mathcal{T}}$, but depending on the shape-regularity of $\tilde{\mathcal{T}}$, on $c_{\Delta}$, and on the polynomial degree $r$.
\end{lemma}

\begin{proof} See Karakashian and Pascal \cite{KP}. 
\end{proof}

The bound \eqref{KP_stab} shows, in particular, that $\norm{\mbf{h}^{-1/2}\jump{v}}{\Gamma_{\dint}}^2$ is a norm on the orthogonal complement $W_h^r$ of $\tilde{V}_h^r\cap H^1(\Omega)$ in $V_h^r$ with respect to the standard $H^1$-inner product.

\begin{remark}
It is possible to make a number of different choices 
\begin{equation}
\mathcal{E}:V_h^r\to V\cap \hat{V}_h^s,
\end{equation}
for instance, for $\hat{V}_h^s$ say a non-conforming finite element
space, e.g., Crouzeix-Raviart elements and $s$ may be in general
different to $r$. Indeed, various choices of $\mathcal{E}$ may give
rise to different methods. As the present work's focus is the
development of conforming methods on polytopic meshes, we prefer to
keep the presentation simple and consider recovery operators into
conforming element-wise polynomials of the same order ($r=s$) but,
crucially, posed on different meshes. For an investigation of the
case $r\neq s$ on standard element shapes, we refer to \cite{RFEM}.
\end{remark}

\section{Design concepts for recovered finite element methods}\label{concept}

Equipped with a finite element space framework and the concept of
recovery operators, we can now describe some general principles in the
design of recovered finite element methods on polytopic meshes.

To this end, we consider a generic conforming Galerkin finite element method for the problem \eqref{pde}, \eqref{bcs}, which is applicable on a \emph{simplicial} mesh, say $\tilde{\mathcal{T}}$, with respective finite element space $\tilde{V}_h\subset V$, reading: find $\tilde{u}_h\in \tilde{V}_h$, such that
\begin{equation}\label{FEM-gen}
a_h(\tilde{u}_h,\tilde{v}_h) = \ell_h(\tilde{v}_h),\qquad\text{ for all } \tilde{v}_h\in\tilde{V}_h;
\end{equation}
an example of a stable such conforming method is the streamline upwind Petrov-Galerkin approach presented and analysed in \cite{Houston2001}.

Suppose also that the bilinear form $a_h$ is coercive in $\tilde{V}_h$ with respect to an ``energy''-like norm $\norm{\cdot}{a}$, i.e., for all $w\in \tilde{V}_h$, there exists a $C_{coer}>0$, such that
\begin{equation}
\label{eq:coer}
C_{coer}\norm{w}{a}^2\le a_h(w,w), 
\end{equation}
and that $a$ is also continuous in $V\times \tilde{V}_h$, in the sense that for all $z\in V$ and all $w\in \tilde{V}_h$, there exists a constant $C_{cont}>0$, such that
\begin{equation}
\label{eq:cont}
| a_h(z,w) | \le C_{cont} \ndg{z}_a \norm{w}{a}, 
\end{equation}
for some norm $\ndg{\cdot}_a$, possibly different to $\norm{\cdot}{a}$. Of course, this is relevant if $\ndg{\cdot}_a$ is stronger than $\norm{\cdot}{a}$ for, otherwise, we can replace $\ndg{\cdot}_a$ by $\norm{\cdot}{a}$ throughout this section. Hence, without any loss of generality in this context, we henceforth assume $\norm{w}{a}\le C\ndg{w}_a$ for all $w\in V$ with $C>0$ independent of $w$. 

A corresponding \emph{recovered finite element method} (R-FEM) can be, then, defined on a \emph{polytopic} mesh $\mathcal{T}$, (with respective finite element space $V_h^r$,) which admits a subtriangulation $\tilde{\mathcal{T}}$, (and a respective space $\tilde{V}_h^r$,) as discussed in detail in Section \ref{sec:rfem_FES} with the help of a recovery operator $\mathcal{E}:V_h^r\to V\cap \tilde{V}_h^r$. To this end, we consider the R-FEM: find $u_h\in V_h^r$ (and, consequently, also $\mathcal{E}(u_h)\in V\cap \tilde{V}_h^r $) such that
\begin{equation}\label{RFEM-gen}
a_h(\mathcal{E}(u_h),\mathcal{E}(v_h)) + s_h(u_h,v_h)= \ell_h(\mathcal{E}(v_h)),\qquad\text{ for all } v_h\in V_h^r,
\end{equation}
for $s_h: \big(\tilde{V}_h^r\cap \prod_{T\in\mathcal{T}}H^1(T)\big)\times\big(\tilde{V}_h^r\cap \prod_{T\in\mathcal{T}}H^1(T)\big)\to \mathbb{R}$ a symmetric bilinear form, henceforth referred to as the \emph{stabilization}, whose role is to remove the possible rank-deficiency due to the use of a recovery operator. Note that $V_h^r\subset\big(\tilde{V}_h^r\cap \prod_{T\in\mathcal{T}}H^1(T)\big)$.

An immediate choice for stabilization can be:
\begin{equation}\label{stab_gen}
s_h(w_h,v_h)= C\int_{\Omega} \mbf{h}^{m}\big(w_h-\mathcal{E}(w_h)\big)\big(v_h-\mathcal{E}(v_h)\big)\ \ud x,
\end{equation}
for $v_h\in V_h^r$, with $m\in \mathbb{R}$ a real number, to be determined by the error analysis in each case. When $\mathcal{E}$ is as in Lemma \ref{lemma_E1}, \eqref{KP_stab} allows also to consider the alternative stabilization
\begin{equation}\label{stab_KP}
s_h(w_h,v_h)= C \int_{\Gamma_{\dint}} \mbf{h}^{m-1}\jump{w_h}\cdot \jump{v_h}\ \ud s.
\end{equation}
To keep the discussion general at this point, we avoid prescribing a specific stabilization, and we prefer to make a structural assumption on $s_h$ instead.
\begin{assumption}\label{stabilisation_ass2}
The stabilization bilinear form satisfies
\[
s_h(w_h,v_h)\le C_{stab}\big(s_h(w_h,w_h)\big)^{1/2}\big(s_h(v_h,v_h)\big)^{1/2} \quad\text{for all }w_h,v_h\in V_h^r,
\] 
for some constant $C_{stab}>0$ independent of $w_h,v_h$ and of $\mbf{h}$.
\end{assumption}
Now given the PDE problem \eqref{pde}, \eqref{bcs} in weak form,
reading: find $u\in V$ such that
\begin{equation}\label{PDE_weak}
a(u,v)=\ell(v) \quad\text{ for all } v\in V.
\end{equation}

We can now show the following best approximation result:

\begin{lemma}
Let $u\in V$ satisfy (\ref{PDE_weak}) and, for $r\geq 1$, suppose
$u_h\in V_h^r$ is the R-FEM approximation with a stabilisation term
satisfying Assumption \ref{stabilisation_ass2} then we have
\begin{equation}
\begin{split}
\frac{1}{2} \norm{u-\mathcal{E}(u_h)}{a}^2
+
C_{coer}^{-1}s_h(u_h,u_h)
&\le
\inf_{v_h\in V_h^r}\bigg(\Big(1+2\frac{C_{cont}^2}{C_{coer}^2}\Big)\ndg{u-\mathcal{E}(v_h)}_a^2
+
\frac{C_{stab}^2}{C_{coer}}s_h(v_h,v_h)\bigg)
\\&\qquad\qquad\qquad\qquad\qquad\qquad\qquad\qquad +
\mathrm{INC}(u),
\end{split}
\end{equation}
where 
\[
\mathrm{INC}(u):=\frac{2}{C_{coer}^2}\sup_{0\neq \tilde{w}_h\in \mathcal{E}(V_h^r)}\bigg( \frac{\big|\ell_h(\tilde{w}_h) -\ell(\tilde{w}_h)\big|}{\norm{\tilde{w}_h}{a}}+\frac{ \big|a( u,\tilde{w}_h)- a_h( u,\tilde{w}_h)\big|}{\norm{\tilde{w}_h}{a}}\bigg)^2,
\]
and $C_{coer}, C_{cont}$ are the coercivity and continuity constants defined in (\ref{eq:coer}), (\ref{eq:cont}) independent of $u$, $u_h$, $h$ and of $\mathcal E$.
\end{lemma}

\begin{proof}
With $\xi:=\mathcal{E}(v_h-u_h)\in \tilde{V}_h^r\cap
V\equiv\tilde{V}_h$, coercivity (\ref{eq:coer}) implies
\[
C_{coer}\norm{\xi}{a}^2+s_h(u_h,u_h)\le a_h(\xi,\xi) + s_h(u_h,u_h),
\]
and hence, in view of \eqref{RFEM-gen}, \eqref{PDE_weak} adding and subtracting appropriate terms yields
\[
\begin{split}
C_{coer}\norm{\xi}{a}^2+s_h(u_h,u_h)
&\leq
a_h(\mathcal{E}(v_h),\xi) -\ell_h(\xi)+s_h(u_h,v_h)\\
&\leq
a_h(\mathcal{E}(v_h) - u,\xi) + s_h(u_h,v_h) +\ell(\xi) - \ell_h(\xi)+ a_h( u,\xi) - a( u,\xi).
\end{split}
\]
Making use of the continuity (\ref{eq:cont}) of $a_h$ along with
and Assumption \ref{stabilisation_ass2} we see
\[
\begin{split}
\norm{\xi}{a}^2
+
\frac{1}{C_{coer}} s_h(u_h,u_h)
&\le
\frac{C_{cont}}{C_{coer}}\ndg{u-\mathcal{E}(v_h)}_a\norm{\xi}{a}
+
\frac{C_{stab}}{C_{coer}}\big(s_h(u_h,u_h)\big)^{1/2}\big(s_h(v_h,v_h)\big)^{1/2}
\\
&\quad
+
\frac{1}{C_{coer}}\big( \ell(\xi) - \ell_h(\xi)+ a_h( u,\xi) - a( u,\xi) \big).
\end{split}
\]
Finally, invoking H\"olders inequality in standard fashion shows the
abstract bound
\begin{equation*}
\begin{split}
\norm{\xi}{a}^2
+
\frac{1}{C_{coer}}s_h(u_h,u_h)
&\le
2\frac{C_{cont}^2}{C_{coer}^2}\ndg{u-\mathcal{E}(v_h)}_a^2
+
\frac{C_{stab}^2}{C_{coer}}s_h(v_h,v_h)+\mathrm{INC}(u).
\end{split}
\end{equation*}
The result follows by the triangle inequality and noticing that $v_h$ was arbitrary.
\end{proof}

To arrive at an a priori error bound, we make the following (rather
mild and immediately satisfiable by all the scenarios we have in mind)
additional set of assumptions.

\begin{assumption}\label{ass:broken_stab}
There exists a ``broken'' version of $\norm{\cdot}{a}$, say
$\norm{\cdot}{a,\mathcal{T}}$, elementwise with respect to
$\mathcal{T}$, for which we have
$\norm{w}{a,\mathcal{T}}=\norm{w}{a}$ whenever $w\in \tilde{V}_h^r\cap
V$. Moreover, for $C_{ker},c_{ker} > 0$ representing constants independent of
$\mbf{h}$ and of $w$ the stabilization $s_h$ satisfies
\[
c_{ker}s_h(w,w)\le \norm{w-\mathcal{E}(w)}{a,\mathcal{T}}^2\le C_{ker}s_h(w,w)\quad \text{ for all } w\in \tilde{V}_h^r\cap \prod_{T\in\mathcal{T}}H^1(T).
\]
That is, this equivalence holds for all elementwise polynomials defined over the simplicial submesh $\mathcal{T}$ that are continuous within in each element $T$ of the related polytopic mesh. 

Finally, assume that there exists a ``broken'' version of
$\ndg{\cdot}_a$, denoted by $\ndg{\cdot}_{a,\mathcal{T}}$, elementwise
with respect to $\mathcal{T}$, for which we have:
\begin{enumerate}
\item
$\ndg{w}_{a,\mathcal{T}}=\ndg{w}_{a}$ whenever $w\in \tilde{V}_h^r\cap V$ and
\item
$\norm{w}{a,\mathcal{T}}\le C\ndg{w}_{a,\mathcal{T}}$ for all $w \in\big(\tilde{V}_h^r\cap \prod_{T\in\mathcal{T}}H^1(T)\big)$ for some $C>0$ independent of $w$ and $\mbf{h}$. 
\end{enumerate}
\end{assumption}

\begin{theorem}\label{thm:abstract}
Assume that the recovery operator $\mathcal{E}$ in the definition of R-FEM \eqref{RFEM-gen} is such that $\mathcal{E}(v)=v$ for all $v\in \tilde{V}_h^r\cap V$ and, also, that it is stable with respect to the $\ndg{w}_{a,\mathcal{T}}$-norm, viz.,
\[
\ndg{\mathcal{E}(w)}_{a}\le C\ndg{w}_{a,\mathcal{T}} \ \ \forall w\in V_h^r.
\]
Assume that the exact solution satisfies $u\in \prod_{T\in\mathcal{T}} H^k(T)$, for some $k\ge 2$, and that any inconsistency of the Galerkin method posed on simplices \eqref{FEM-gen} is of optimal order, viz.,
\[
\mathrm{INC}(u)\le C\sum_{T\in\tilde{\mathcal{T}}}\mbf{h}^{2s}|u|_{s,T}^2,
\]
for $1\le s=\min\{k,r\}$, with constant $C>0$, independent of $u$ and $\mbf{h}$. Then, we have the bound
\begin{equation}
\norm{u-\mathcal{E}(u_h)}{a}^2+s_h(u_h,u_h)
\le C\sum_{T\in\tilde{\mathcal{T}}}\mbf{h}^{2s}|u|_{s,T}^2.
\end{equation}
\end{theorem}

\begin{proof}
The triangle inequality, Assumption \ref{ass:broken_stab} along with
the optimality of the inconsistency terms imply
\begin{equation}
\norm{u-\mathcal{E}(u_h)}{a}^2+s_h(u_h,u_h)
\le C \ndg{u-\Pi u}_{a,\mathcal{T}}^2 
+ Cs_h(\Pi u, \Pi u)+C\sum_{T\in\tilde{\mathcal{T}}}\mbf{h}^{2s}|u|_{s,T}^2,
\end{equation}
with $\Pi:L^2(\Omega)\to V_h^r$ denoting the orthogonal
$L^2$-projection operator onto the polytopic finite element space
$V_h^r$. Let also $\tilde{\Pi}:L^2(\Omega)\to \tilde{V}_h^r\cap V$ be
the respective orthogonal $L^2$-projection onto the conforming finite
element space of the subtriangulation $\tilde{\mathcal{T}}$. The
above mesh assumptions on $\mathcal{T}$ and on $\tilde{\mathcal{T}}$
ensure that $\Pi$ and $\tilde{\Pi}$ admit optimal approximation
properties.

From hypothesis, we have
$\mathcal{E}(\tilde{\Pi}u)=\tilde{\Pi}u$. From Assumption
\ref{ass:broken_stab}, it then follows $s_h(\tilde{\Pi} u, \tilde{\Pi}
u)=0$. Now, from Assumption \ref{stabilisation_ass2}, we then also
have
\[
s_h(\tilde{\Pi} u,v)=s_h(v,\tilde{\Pi} u)=0 \text{ for any } v\in
\tilde{V}_h^r\cap \prod_{T\in\mathcal{T}}H^1(T).
\]
Using this,
together with Assumption \ref{ass:broken_stab} and the stability of
the recovery operator, we have, respectively,
\[
\begin{split}
s_h(\Pi u, \Pi u)
&=
s_h(\tilde{\Pi}u-\Pi u, \tilde{\Pi}u-\Pi u)
\\
&\le 
c_{ker}^{-1} \norm{\tilde{\Pi}u-\Pi u-\mathcal{E}(\tilde{\Pi}u-\Pi u)}{a,\mathcal{T}}^2
\\
&\le
C\big( \norm{\tilde{\Pi}u-\Pi u}{a,\mathcal{T}}^2+\ndg{\tilde{\Pi}u-\Pi u}_{a,\mathcal{T}}^2\big)
\\
&\le
C\ndg{\tilde{\Pi}u-\Pi u}_{a,\mathcal{T}}^2
\\
&\le
C\big( \ndg{u-\tilde{\Pi}u}_{a,\mathcal{T}}+ \ndg{u-\Pi u}_{a,\mathcal{T}}\big)^2.
\end{split}
\] 
The result now follows by appealing to the optimal approximation properties of $\Pi$ and $\tilde{\Pi}$.
\end{proof}

\begin{corollary}
With the assumptions of Theorem \ref{thm:abstract}, we also have the following bound:
\begin{equation}
\norm{u-u_h}{a,\mathcal{T}}^2+s_h(u_h,u_h)
\le C\sum_{T\in\tilde{\mathcal{T}}}\mbf{h}^{2s}|u|_{s,T}^2,
\end{equation}
for $1\le s=\min\{k,r\}$, with $C$ positive constant, independent of $u$ and of $\mbf{h}$.
\end{corollary}
\begin{proof}
The triangle inequality implies
\[
\norm{u-u_h}{a,\mathcal{T}}\le \norm{u-\mathcal{E}(u_h)}{a}+\norm{\mathcal{E}(u_h)-u_h}{a,\mathcal{T}}.
\]
Using, now, Assumption \ref{ass:broken_stab} we have
\[\norm{\mathcal{E}(u_h)-u_h}{a,\mathcal{T}}\le C_{ker}s_h(u_h,u_h),\]
the result then follows.
\end{proof}

\section{An alternative recovered finite element method}
\label{sec:alternative}
To highlight further the potential of the proposed R-FEM framework
applied to both standard/simplicial/box and, in general, polytopic
meshes, we present an alternative R-FEM. This method is motivated by
the desire to have a conforming approximation for the second order
part of the differential operator and an upwinded discontinuous
Galerkin discretization of the first order terms in \eqref{pde}. The
developments below also showcase an R-FEM error analysis using inf-sup
stability rather than coercivity results.

For sake of the simplicity of exposition, we assume that each entries
of the diffusion tensor $a$ are constant on each element
$T\in\mathcal{T}$,i.e.,
\begin{equation}\label{diffusion_tensor}
a\in [V_h^0]^{d\times d}_{\text{sym}}.
\end{equation}
Additionally, we assume the following standard assumption on $\mbf{b}$:
\begin{equation}\label{convection_field}
\mbf{b} \cdot \nabla_h \xi \in V_h^r \qquad \forall \xi \in V_h^r. 
\end{equation}
cf. \cite{hss,DGpoly2}. For given operator \eqref{recovery}, the new
\emph{recovered finite element method} reads: find $u_h\in V_h^r$ such
that
\begin{equation}\label{rem}
B(u_h,v_h)= \ell (v_h),\qquad \text{ for all }v_h\in V_h^r,
\end{equation}
where
\begin{equation}\label{method}
\begin{split}
B(u_h,v_h) &:=
\int_{\Omega} \Big( a\nabla \mathcal{E}(u_h)\cdot \nabla \mathcal{E}(v_h)
+
\mbf{b}\cdot\nabla_h u_h\mathcal{E}(v_h)
+
c\mathcal{E}(u_h)\mathcal{E}(v_h)\Big)\ud x\\
&\quad
-
\int_{\Gamma_{\ddd}} \!\Big( a \nabla \mathcal{E}(u_h)\cdot \mbf{n}\, \mathcal{E}(v_h)
+
a \nabla \mathcal{E}(v_h)\cdot \mbf{n}\, \mathcal{E}(u_h)
-
\sigma_{\ddd} \mathcal{E}(u_h)\mathcal{E}(v_h)\Big)\ud s\\ 
&\quad
-
\int_{\Gamma_{\ddd}^-} \big(\mbf{b}\cdot\mbf{n} \big) u_h \mathcal{E}(v_h)\ud s
-
\sum_{T\in\mathcal{T}}\int_{\partial_-T\backslash\partial\Omega} \big(\mbf{b}\cdot\mbf{n} \big)
\lfloor u_h\rfloor \mathcal{E}(v_h)\ud s\\
&\quad
+
s_h^{a,c}(u_h,v_h)
+
s_h^b(u_h,v_h),
\end{split}
\end{equation}
and
\[
\begin{split}
\ell (v_h)
&:=
\int_{\Omega} f \mathcal{E}(v_h)\,\ud x
-
\int_{\Gamma_{\ddd}} g_{\ddd}^{} \Big( a \nabla \mathcal{E}(v_h)\cdot \mbf{n}
-
\sigma_{\ddd} \mathcal{E}(v_h)\Big)\ud s
+
\int_{\Gamma_{\dn}}g_{\dn}^{}\, \mathcal{E}(v_h)\ud s\\
&\quad
-
\int_{\Gamma_{\ddd}^-} \big(\mbf{b}\cdot\mbf{n} \big) g_{\ddd}^{} \mathcal{E}(v_h)\,\ud s, 
\end{split}
\]
with $s_h^m(\cdot,\cdot):V_h^r\times V_h^r\to \mathbb{R}$,
$m\in\{\{a,c\},b\}$ denoting symmetric bilinear forms, henceforth
referred to as \emph{stabilisations}, and
$\sigma_{\ddd}:\Gamma_{\ddd}\to \mathbb{R}$ a positive penalty
function defined precisely below that weakly enforces the Dirichlet
boundary conditions.

This motivates the following choice for the elliptic stabilisation bilinear form:
\begin{equation}\label{stabilisationa}
s_h^{a,c}(u_h,v_h):= \int_{\Gamma_{\dint}}\sigma_{a,c} \jump{u_h}\cdot\jump{v_h}\,\ud s,
\end{equation}
for some non-negative function
$\sigma_{a,c}:\Gamma_{\dint}\to\mathbb{R}$, that will also be defined
below.

To ensure that sufficient numerical diffusion is included in the
proposed method for the case of small or vanishing diffusion tensor
$a$, we select
\begin{equation}\label{stabilisationb}
s_h^b(u_h,v_h):= \int_{\Gamma_{\dint}}\Big( \sigma_{b,1} 
\jump{u_h}\cdot\jump{v_h}
+
\sigma_{b,2} \jump{\mbf{h}(\mbf{b}\cdot\nabla u_h)}\cdot\jump{\mbf{h}(\mbf{b}\cdot\nabla v_h)}\Big)\ud s,
\end{equation}
for non-negative functions
$\sigma_{b,1},\sigma_{b,2}:\Gamma_{\dint}\to\mathbb{R}$, to be
selected below. We note that \eqref{stabilisationb} follows the spirit
of the, so-called, continuous interior penalty stabilisation procedure
due to Douglas and Dupont \cite{CIP_old} and to Burman and Hansbo
\cite{CIP}. Crucially, however, the trial and test functions $u_h$ and
$v_h$ in the present R-FEM context are discontinuous, cf.,
\cite{MR2192329} also. The inconsistency introduced by
the streamline derivative jump term in \eqref{stabilisationb} will be
dealt with in the a priori error analysis below.

Some remarks on the method are in order. To accommodate for the
potentially locally changing nature of the differential operator, we
have opted for weak imposition of essential boundary conditions,
following the classical ideas from \cite{ nits,reedhill}. For the
case of elliptic problems, strong imposition of essential boundary
conditions in the spirit of \cite{RFEM} is by all means possible
also. We note, however, that since essential boundary values are known
the above is actually a conforming method for $\mathcal{E}(u_h)$.

Also, we have opted for a, to the best of our knowledge, new method
for the discretisation of the first order term. This is to highlight the
flexibility of RFEM in incorporating different discretisations of
various terms of the differential operator. More importantly, since in
the absence of diffusion, the exact solution $u$ may exhibit jump
discontinuities across characteristic surfaces, we prefer not to
recover $u_h$ in the discretisation of the first order term. We stress
that more classical choices, such as streamline diffusion-type and/or
continuous interior penalty-type treatment of the first order term are
by all means possible. Indeed, certain such choices coincide with the
standard/classical conforming finite element versions for
$\mathcal{E}(u_h)$ when applied to standard triangular meshes (cf.,
the discussion in Section 3 of \cite{RFEM}).

We also remark on the assumptions on the diffusion tensor
\eqref{diffusion_tensor} and convection field
\eqref{convection_field}. The above RFEM method \eqref{rem} can be easily
extended to general positive semi-definite diffusion $a\in
[L^\infty(\Omega)]^{d\times d}_{\text{sym}}$ following the
inconsistent formulation introduced in \cite{gl}. For the general
convection field $\mbf{b}$, we would need to modify
\eqref{stabilisationb} by setting
\begin{equation}
s_h^b(u_h,v_h)
= \int_{\Gamma_{\dint}}
\Big( \sigma_{b,1} 
\jump{u_h}\cdot\jump{v_h}
+
\sigma_{b,2} \jump{\mbf{h}\Pi(\mbf{b}\cdot\nabla u_h)}\cdot\jump{\mbf{h}\Pi(\mbf{b}\cdot\nabla v_h)}\Big)\ud s,
\end{equation}
which will make the stability proof and error analysis more
complicated. We refrain from doing this here to focus on the key
ideas.

\section{A priori error analysis}
\label{sec:apriori}
We dedicate this section to the analysis of the method introduced in
Section \ref{sec:alternative}. The main ingredient to this is an
inf-sup condition over suitable norms. We let $\alpha, \beta, \gamma:
\Omega \to \mathbb{R}$ such that
\begin{equation}
\alpha|_T := |\sqrt{a}|^2_2 |_T,
\quad
\beta|_T := \norm{\mbf{b}}{L^{\infty}(T)},
\quad
\gamma|_T :=\norm{c}{L^{\infty}(T)},
\end{equation}
over each element $T\in \mathcal{T}$. We define the stabilisation
parameter
\begin{equation}
\sigma_{\ddd}:=C_{\sigma} \alpha r^2/\mbf{h},
\quad
\tilde{\sigma}_{\ddd}|_T := \max_{e\subset \partial T}\sigma_{\ddd}|_e.
\end{equation}
Now, for $w\in V_h^r$, we define the norms
\[
\begin{split}
\norm{w}{\db}
:=\Big(
\norm{\sqrt{c_0} w}{}^2
+
\frac{1}{2}
\big(\norm{\sqrt{|\mbf{b}\cdot\mbf{n}|} \lfloor w\rfloor }{\Gamma_{\dint}}^2
+
\norm{\sqrt{|\mbf{b}\cdot\mbf{n}|} w}{\Gamma_{\ddd}^-}^2
+
\norm{\sqrt{|\mbf{b}\cdot\mbf{n}|} w }{\Gamma_{\dn}^{+}}^2
\big)\Big)^{1/2},
\end{split}
\]
and
\[
\begin{split}
\ndg{w}
&:=\Big(
\norm{\sqrt{a}\nabla \mathcal{E}(w)}{}^2
+
\norm{\sqrt{\sigma_{\ddd}} \mathcal{E}(w)}{\Gamma_{\ddd}}^2
+
\norm{w}{\db}^2
+
s_h^{a,c}(w,w)
+
s_h^{b}(w,w)\Big)^{1/2}.
\end{split}
\]
We also define
the `streamline-diffusion' norm
\[
\nsdg{w}:=\big( \ndg{w}^2+\norm{\sqrt{\delta \mbf{\lambda}} (\mbf{b} \cdot \nabla_h w)}{}^2\big)^{1/2},
\]
where $\delta>0$, to be chosen precisely below, and
\begin{equation}\label{spara}
\mbf{\lambda} := \min \{\beta^{-1}, \tilde{\sigma}_{\ddd}^{-1} \}\mbf{h}, 
\end{equation}
We are now in a position to show the following inf-sup condition for the R-FEM method (\ref{rem}). In the proofs of the results in this
section, we are particularly interested in the dependence of the
resulting bounds on the mesh-P\'eclet number $Pe_h$, the mesh-size
$\mbf{h}$ and polynomial degree $r$. We aim, therefore, to track
constants and their dependence explicitly.
\begin{theorem}[Inf-Sup Condition]
\label{lem:inf-sup}
Let $a\in[L^\infty(\Omega)]^{d\times d}_{\text{sym}}$ satisfy
Assumption (\ref{diffusion_tensor}) and $\mbf{b}\in
W^{1,\infty}(\Omega)^d$ satisfy Assumption \eqref{convection_field}.
Assume that the mesh is such that each element face in the mesh is
either completely inflow or outflow or characteristic. Suppose also
that the penalisation parameters $\sigma_{a,c}, \sigma_{b,1}$ and
$\sigma_{b,2}$ are chosen large enough to satisfy (\ref{eq:simgas}), $\delta$ is chosen to satisfy (\ref{def-sigma})
and the boundary stabilisation constant $C_{\sigma}>0$ is
sufficiently large. Then, we have
\begin{equation}
\inf_{0\neq w_h\in V_h^r}\sup _{0\neq v_h\in V_h^r}
\frac{B(w_h,v_h)}{\ndg{w_h}_{s}\ndg{v_h}_s}
\ge
\Lambda,
\end{equation}
where $\Lambda > 0$ is independent of $\mbf{\lambda}$, $\mbf{h}$ and of the
mesh-P\'eclet number $Pe_h:=\beta \mbf{h}/\alpha$.
\end{theorem}

\begin{proof}
As usual, the proof consists of two steps: 1) for each $w_h\in V_h^r$, we find a $v_h(w_h)\equiv v_h\in V_h^r$ such that
$
B(w_h, v_h) \geq C \nsdg{w_h}^2,
$
and, 2) we show that this $v_h$ satisfies the bound
$
\nsdg{v_h} \leq C \nsdg{w_h},
$
thereby inferring the result. 

To that end, fix $w_h\in V_h^r$ and set $v_h=w_h+\delta w_h^b$,
where we will use the shorthand $w_h^{b}:=\mbf{\lambda}
(\mbf{b}\cdot \nabla_h w_h)$ for brevity for some $\delta \in
\reals$ is to be chosen. Then, integration by parts and working as in the proof of \cite[Lemma 2.4]{HSS_hyp}, as well as making use of standard inverse estimates, give
\begin{equation}\label{eq:mid_one}
\begin{split}
B(w_h,w_h)
&\ge
\frac{1}{2}\ndg{w_h}^2
+
\int_{\Omega}\mbf{b}\cdot \nabla_h w_h\big(\mathcal{E}(w_h)-w_h\big)\ud x
+
\int_{\Omega} c\big( \mathcal{E}(w_h)^2 - w_h^2\big)\ud x\\
&\quad 
-
\int_{\Gamma_{\ddd}^-} \big(\mbf{b}\cdot\mbf{n} \big) w_h\big(\mathcal{E}(w_h)-w_h\big) \ud s
-
\sum_{T\in\mathcal{T}}\int_{\partial_-T\backslash\partial\Omega} \big(\mbf{b}\cdot\mbf{n} \big) \lfloor w_h\rfloor \big(\mathcal{E}(w_h)-w_h\big)\ud s\\
&=: \frac{1}{2}\ndg{w_h}^2 + \rom{1} + \rom{2} + \rom{3} + \rom{4}.
\end{split}
\end{equation}
Using Lemma \ref{lemma_E1},
Young's inequality and \eqref{spara}, we have
\[
\begin{split}
\rom{1}
&\le
\frac{1}{4}\norm{\sqrt{\delta\mbf{\lambda}}\mbf{b}\cdot\nabla_h w_h}{}^2
+
C(r)\norm{ \delta^{-1/2} \beta^{1/2}\jump{w_h}}{\Gamma_{\dint}}^2.
\end{split}
\]
and
\[
\begin{split}
\rom{2}
&=
\int_{\Omega} c\big( \mathcal{E}(w_h) - w_h\big)^2+2cw_h\big( \mathcal{E}(w_h) - w_h\big)\ud x
\\
& \le
C(r)\norm{\sqrt{\gamma\mbf{h}}\jump{w_h}}{\Gamma_{\dint}}^2 +\frac{1}{4}\norm{\sqrt{c_0}w_h}{}^2
+
C(r)\norm{\sqrt{\gamma/\underline{\gamma}_0}\sqrt{\gamma\mbf{h}}\jump{w_h}}{\Gamma_{\dint}}^2
\\
&\le
\frac{1}{4}\norm{\sqrt{c_0}w_h}{}^2 +
C(r)\norm{\big(1+\sqrt{\gamma/\underline{\gamma}_0}\big)\sqrt{\gamma\mbf{h}}\jump{w_h}}{\Gamma_{\dint}}^2,
\end{split}
\]
with $\underline{\gamma}_0$ a local min of $c_0$. Finally, again,
using Lemma \ref{lemma_E1} and Young's inequality, we also have
\[
\begin{split}
\rom{3} + \rom{4}
&\le
\frac{1}{8}\norm{\sqrt{|\mbf{b}\cdot\mbf{n}|} w_h}{\Gamma_{\ddd}^-}^2
+
\frac{1}{8}\norm{\sqrt{|\mbf{b}\cdot\mbf{n}|} \lfloor w_h\rfloor}{\Gamma_{\dint}}^2
+
C(r)\norm{\sqrt{\beta}\jump{w_h}}{\Gamma_{\dint}}^2.
\end{split}
\]
Substituting the above bounds into \eqref{eq:mid_one}, we arrive at
\begin{equation}\label{eq:mid_one_two}
\begin{split}
B(w_h,w_h) &\ge \frac{1}{2}\ndg{w_h}^2-\frac{1}{4}\norm{\sqrt{c_0}w_h}{}^2 -
C(r)\norm{\big(1+\sqrt{\gamma/\underline{\gamma}_0}\big)\sqrt{\gamma\mbf{h}}\jump{w_h}}{\Gamma_{\dint}}^2\\ &\quad -\frac{1}{8}\norm{\sqrt{|\mbf{b}\cdot\mbf{n}|} w_h}{\Gamma_{\ddd}^-}^2- \frac{1}{8}\norm{\sqrt{|\mbf{b}\cdot\mbf{n}|} \lfloor w_h\rfloor}{\Gamma_{\dint}}^2-\frac{1}{4}\norm{\sqrt{\delta\mbf{\lambda}}\mbf{b}\cdot\nabla_h w_h}{}^2\\
&\quad - C(r)\norm{(\sqrt{\beta}+\delta^{-1/2}\sqrt{\beta})\jump{w_h}}{\Gamma_{\dint}}^2.
\end{split}
\end{equation}
Working as before, we also have
\begin{equation}\label{eq:second_bound}
\begin{split}
B(w_h,\delta w_h^b) 
&\ge - \frac{1}{4} \ndg{w_h}
-
\norm{\sqrt{a}\nabla \mathcal{E}(\delta w_h^b)}{}^2\\
&\quad +
\norm{\sqrt{\delta \mbf{\lambda}} \mbf{b}\cdot \nabla_h w_h}{}^2
+
\int_{\Omega}\mbf{b}\cdot \nabla_h w_h\big(\mathcal{E}(\delta w_h^b) - \delta w_h^b\big)\ud x\\
&\quad
-
2\norm{\sqrt{\sigma_{\ddd}} \mathcal{E}(\delta w_h^b)}{\Gamma_{\ddd}}^2
-
8\norm{\sigma_{\ddd}^{-1/2} a\nabla \mathcal{E}(\delta w_h^b)}{\Gamma_{\ddd}}^2\\
&\quad
-
2\norm{\sqrt{|\mbf{b}\cdot\mbf{n}|} \mathcal{E}(\delta w_h^b)}{\Gamma_{\ddd}^-}^2
-
2\norm{\sqrt{|\mbf{b}\cdot\mbf{n}|} \mathcal{E}(\delta w_h^b)}{\Gamma_{\dint}}^2\\
&\quad
-
\norm{\sqrt{\sigma_{a,c}}\jump{\delta w_h^b}}{\Gamma_{\dint}}^2
-
\norm{\sqrt{\sigma_{b,1}}\jump{\delta w_h^b}}{\Gamma_{\dint}}^2
-
\norm{\sqrt{\sigma_{b,2}}\jump{\mbf{h}(\mbf{b}\cdot \nabla(\delta w_h^b)}}{\Gamma_{\dint}}^2.
\end{split}
\end{equation}
We further bound each term in (\ref{eq:second_bound}) not directly
appearing in the energy norm. We have
\begin{equation}\label{eq:mid1}
\begin{split}
\norm{\sqrt{a}\nabla \mathcal{E}(\delta w_h^b)}{}^2 
&\le
Cr^4 \norm{\sqrt{\alpha}\mbf{h}^{-1}\mathcal{E}(\delta w_h^b)}{}^2\\
&\le
Cr^4 \norm{ \sqrt{\delta} (\sqrt{ \delta \mbf{ \lambda}} \mbf{b} \cdot \nabla_h w_h)}{}^2
+
C(r)\norm{ \jump{\delta \mbf{h} (\mbf{b} \cdot \nabla_h w_h)}}{\Gamma_{\dint}}^2,
\end{split}
\end{equation}
using an inverse estimate
and Lemma \ref{lemma_E1}, respectively. Similarly,
\[
\begin{split}
\int_{\Omega}\mbf{b}\cdot \nabla_h w_h\big(\mathcal{E}(\delta w_h^b) - \delta w_h^b\big)\ud x
& \le
\frac{1}{4}\norm{\sqrt{\delta \mbf{\lambda}} \mbf{b}\cdot \nabla_h w_h}{}^2
+
C(r)\norm{ \jump{\sqrt{\delta} \mbf{h} (\mbf{b} \cdot \nabla_h w_h)}}{\Gamma_{\dint}}^2.
\end{split}
\]
Next, using the
stability of $\mathcal{E}$, we deduce
\[
\begin{split}
2\norm{\sqrt{\sigma_{\ddd}} \mathcal{E}(\delta w_h^b)}{\Gamma_{\ddd}}^2
&
\le Cr^4\norm{ \sqrt{\delta} (\sqrt{\delta \mbf{\lambda}}\mbf{b}\cdot \nabla_h w_h)}{}^2
+
C(r)\norm{ \jump{\delta \mbf{h} (\mbf{b} \cdot \nabla_h w_h)}}{\Gamma_{\dint}}^2,
\end{split}
\]
and, using inverse estimates along with the definition of $\sigma_{\ddd}$,
\[
\begin{split}
8\norm{\sigma_{\ddd}^{-1/2} a\nabla \mathcal{E}(\delta w_h^b)}{\Gamma_{\ddd}}^2&
\le Cr^4\norm{\sqrt{\alpha} \mbf{h}^{-1}\mathcal{E}(\delta w_h^b)}{}^2,
\end{split}
\]
which can be further estimated as in \eqref{eq:mid1}. We also have
\[
2 \norm{\sqrt{|\mbf{b}\cdot\mbf{n}|} \mathcal{E}(\delta w_h^b)}{\Gamma_{\ddd}^-\cup\Gamma_{\dint}}^2
\le 
Cr^2\norm{ \sqrt{\delta} (\sqrt{\delta \mbf{\lambda}} \mbf{b}\cdot\nabla_h w_h)}{}^2
+
C(r)\norm{ \jump{ \delta \mbf{h} (\mbf{b} \cdot \nabla_h w_h)}}{\Gamma_{\dint}}^2,
\]
and
\[
\norm{\sqrt{\sigma_{b,2}}\jump{\mbf{h}(\mbf{b}\cdot \nabla_h(\delta w_h^b)}}{\Gamma_{\dint}}^2
\le
Cr^6\norm{\sqrt{\sigma_{b,2}/\mbf{h}}\beta \delta w_h^b}{}^2
\le Cr^6\norm{\sqrt{\sigma_{b,2} \beta \delta} (\sqrt{\delta \mbf{\lambda}} \mbf{b}\cdot\nabla_h w_h)}{}^2.
\]
Combining now the above estimates, we arrive at the bound
\begin{equation}\label{eq:fourth_bound}
\begin{split}
B(w_h, v_h(w_h)) 
&\ge
\frac{1}{4}\ndg{w_h}^2
+
\frac{1}{2}\norm{\sqrt{\delta\mbf{h}}\mbf{b}\cdot\nabla_h w_h}{}^2\\
&\quad
-
C(r)\norm{(\sqrt{\beta}+\sqrt{\beta}\delta^{-1/2})\jump{w_h}}{\Gamma_{\dint}}^2
-
C(r)\norm{\big(1+\sqrt{\gamma/\underline{\gamma}_0}\big)\sqrt{\gamma\mbf{h}}\jump{w_h}}{\Gamma_{\dint}}^2\\
&\quad
-
Cr^4 \norm{\sqrt{\delta }\Big(1+r^{-1}+r\sqrt{\sigma_{b,2} \beta} \Big)(\sqrt{\delta \mbf{h}} \mbf{b}\cdot \nabla_h w_h)}{}^2
\\
&\quad
-
C(r) \norm{\Big(\sqrt{\sigma_{b,1}}+ 1+\delta ^{-1/2}\Big)\jump{\delta \mbf{h}(\mbf{b}\cdot \nabla_h w_h)}}{\Gamma_{\dint}}^2.
\end{split}
\end{equation}
Upon selecting a global constant $\delta>0$, small enough, we can have 
\[
Cr^4 \delta \Big(1+r^{-1}+r\sqrt{\sigma_{b,2} \beta} \Big)^2\le\frac{1}{4}
\qquad \text{and}\qquad
\delta^2 C(r) \Big(\sqrt{\sigma_{b,1}}+1+\delta^{-1/2}\Big)^2 \le \frac{\sigma_{b,2}}{8}.
\]
Now, upon selecting additionally penalty parameters large enough to satisfy
\begin{equation}\label{eq:simgas}
\begin{split}
\sigma_{b,1}\ge
8C(r)\beta(1+\delta^{-1/2})^2, \quad 
\sigma_{a,c}\ge 
8C(r) \gamma \big(1+\sqrt{\gamma/\underline{\gamma}_0}\big)^2 , \quad
\sigma_{b,2} > 0,
\end{split}
\end{equation}
we can set
\begin{equation}\label{def-sigma}
\delta := \min\{ \Big( 4Cr^4(1+r^{-1}+r\sqrt{\sigma_{b,2}\beta})^2\Big)^{-1}, \Big( -1+\sqrt{1+\sqrt{\sigma_{b,2}}(2\sqrt{C(r)\beta}+2^{-1/2})^{-1}})^2\Big)/{4} \}.
\end{equation}
which, in turn implies,
\begin{equation}\label{infsup_one}
B(w_h, v_h(w_h))
\ge
\frac{1}{8}\ndg{w_h}^2+\frac{1}{4}\norm{\sqrt{\delta\mbf{h}}(\mbf{b}\cdot\nabla_h w_h)}{}^2
\ge
\frac 1 8 \nsdg{w_h}^2,
\end{equation}
from \eqref{eq:fourth_bound} and the first step of the proof is complete.

For the second step, working as above, standard inverse estimates along with
and Lemma \ref{lemma_E1} imply
\[
\begin{split}
\ndg{\delta w_h^b}^2
&\le
C(r)\norm{ \jump{ \delta \mbf{h} (\mbf{b} \cdot \nabla_h w_h)}}{\Gamma_{\dint}}^2
+
Cr^4\norm{\sqrt{\delta } (\sqrt{\delta\mbf{\lambda}}\mbf{b}\cdot \nabla_h w_h )}{}^2
+
\norm{\sqrt{c_0 \delta \mbf{\lambda} } \sqrt{\delta\mbf{\lambda}}\mbf{b}\cdot \nabla_h w_h }{}^2\\
&\quad
+
\norm{\sqrt{\sigma_{a,c}} \jump{\delta w_h^b}}{\Gamma_{\dint}}^2
+
Cr^2\norm{\sqrt{\beta\delta}\sqrt{\delta\mbf{\lambda}}\mbf{b}\cdot\nabla_h w_h }{\Gamma_{\dint}}^2
+
\norm{\sqrt{\sigma_{b,1}}\jump{\delta w_h^b}}{\Gamma_{\dint}}^2\\
&\quad
+
Cr^6\norm{\sqrt{\sigma_{b,2}\delta}\beta\sqrt{\delta\mbf{\lambda}}\mbf{b}\cdot\nabla_h w_h}{}^2
\end{split}
\]
and
$ \norm{\sqrt{\delta\mbf{\lambda}}\mbf{b}\cdot\nabla_h \delta w_h^b}{}^2 
\le
Cr^4\norm{\delta \sqrt{\delta\mbf{\lambda}}\mbf{b}\cdot\nabla_h w_h}{}^2.
$
The above assumptions on $\delta$, $\sigma_{a,c}$, $\sigma_{b,1}$,
$\sigma_{b,2}$, finally imply
\begin{equation}
\label{infsup_two}
\begin{split}
\nsdg{v_h}
&\leq 
\nsdg{w_h} + \nsdg{\delta w_h^b} \le
C(r) \big(\ndg{ w_h}+\norm{\sqrt{\delta\mbf{h}}\mbf{b}\cdot\nabla_h w_h}{}\big)
=
C(r)\nsdg{w_h},
\end{split}
\end{equation}
thereby completing the proof of the second step also.\end{proof}

\begin{remark}
It is possible to modify the above proof by introducing a locally
variable $\delta$ aiming to achieve stronger streamline-diffusion
stabilization effect at the expense of mild assumptions on the local
variation of $\delta$ in the computational domain, see \cite{SINUM} for a
similar argument in a completely different context. In particular, using the elementary identity $
\jump{\delta w_h^b}= \jump{\delta}\av{w_h^b}+\av{\delta}
\jump{w_h^b} $ valid on every face $e\subset \Gamma_{\dint}$,
provided that $\norm{\jump{\delta}}{L^{\infty}(e)}/
\norm{\delta}{L^{\infty}(e)}\ll 1$, one can incorporate $\av{w_h^b}$
into the stabilization term
$\norm{\sqrt{\delta\mbf{\lambda}}\mbf{b}\cdot\nabla_h w_h}{}$.
\end{remark}

\begin{proposition}[Galerkin orthogonality]
\label{prop:gal-orthog}
Let $u\in V$ be the solution of (\ref{pde}),
(\ref{bcs}). Suppose also that $u_h \in V_h^r$ is the R-FEM solution
of (\ref{rem}) and set $e:=u-\mathcal{E}(u_h)$ for brevity. Then, for all $v_h \in V_h^r$ we have:
\begin{equation}\label{galerkin-orthognality}
\begin{split}
&\int_{\Omega} \Big( a\nabla e \cdot \nabla \mathcal{E}(v_h) +\mbf{b}\cdot\nabla_h \big(u- u_h \big)\mathcal{E}(v_h)+c e\mathcal{E}(v_h)\Big)\ud x\\
&\quad - \int_{\Gamma_{\ddd}} \!\Big( a \nabla e\cdot \mbf{n}\, \mathcal{E}(v_h)+ a \nabla \mathcal{E}(v_h)\cdot \mbf{n}\, e-\sigma_{\ddd} e \mathcal{E}(v_h)\Big)\ud s\\ 
&\quad -\int_{\Gamma_{\ddd}^-} \big(\mbf{b}\cdot\mbf{n} \big) \big(u-u_h\big) \mathcal{E}(v_h)\ud s
-\sum_{T\in\mathcal{T}}\int_{\partial_-T\backslash\partial\Omega} \big(\mbf{b}\cdot\mbf{n} \big) \lfloor u- u_h\rfloor \mathcal{E}(v_h)\ud s\\
&\quad+ s_h^{a,c}(u-u_h,v_h)+ s_h^b
(u-u_h,v_h) = 0.
\end{split}
\end{equation}
\end{proposition}

\begin{proof}
To begin, from \eqref{pde}, \eqref{bcs}, the consistency of the method yields
\begin{equation}
\label{pde_weak}
\begin{split}
\int_{\Omega} a\nabla u\cdot\nabla \mathcal{E}(v_h)
+
\mbf{b}\cdot \nabla u\,\mathcal{E}(v_h)
+&
c u\, \mathcal{E}(v_h) \ud x
-
\int_{ \Gamma_{\ddd}}a\nabla u\cdot\mbf{n}\,\mathcal{E}(v_h)\ud s \\
&=
\int_{\Omega}f \mathcal{E}(v_h)\ud x
+
\int_{\Gamma_{\dn}}g_{\dn}^{}\mathcal{E}(v_h)\ud s
\quad \forall v_h\in V_h^r.
\end{split}
\end{equation}
Noting that $\mbf{n}^Ta \mbf{n}=0$ implies $\mbf{n}^Ta=\mbf{0}^T$
for $a$ satisfying \eqref{nonneg}. Moreover, the regularity of $u$ \cite{OR} allows to also conclude
\[
s_h^{a,c}(u,v_h) = s_h^b (u,v_h) = 0 \text{ for } v_h\in V_h^r,
\]
Subtracting now \eqref{rem} from
\eqref{pde_weak} already yields the result.
\end{proof}

\begin{lemma}
\label{lem:splitty}
Suppose the assumptions of Theorem \ref{lem:inf-sup} hold and let
$u$ and $u_h$ satisfy the assumptions of Proposition
\ref{prop:gal-orthog}. In addition, let $\Pi:L^2(\Omega)\to V_h^r$
denote the orthogonal $L^2$-projection operator onto the polytopic
finite element space $V_h^r$. Upon considering the splitting
$
u - u_h = (u - \Pi u) - (u_h - \Pi u) =:\eta - \xi ,
$
we have
\begin{equation}
\label{bilinear}
\begin{split}
|B(\xi, v_h)|
&\leq
F(\eta)
\nsdg{v_h},
\end{split}
\end{equation}
where
\[
\begin{split}
F(\eta)^2
&= 
2\norm{\sqrt{a}\nabla \big(u-\mathcal{E}(\Pi u)\big)}{}^2
+
8C(r)/\sigma_{b,1} \norm{\sqrt{\mbf{h}} \mbf{b} \cdot \nabla_h \eta}{}^2
+
8C(r)\beta /\sigma_{b,1} \norm{\sqrt{|\mbf{b}\cdot\mbf{n}|}\eta }{\Gamma_{\ddd}^-}^2
\\ 
&\qquad
+
8C(r)\beta/ \sigma_{b,1} \norm{\sqrt{|\mbf{b}\cdot\mbf{n}|} \lfloor \eta \rfloor}{\Gamma_{\dint}}^2
+
2\norm{\sqrt{|\mbf{b}\cdot\mbf{n}|}\eta}{\Gamma_{\dn}^+}^2
+
2 \norm{\sqrt{|\mbf{b}\cdot\mbf{n}|}\eta^-}{\Gamma_{\dint}}^2 \\ 
&\qquad
+
(2\norm{\mbf{b}}{W^{1,\infty}(\Omega)}^2/\gamma^2_0)  \norm{\sqrt{c_0}\eta}{}^2
+
\big(1+2\gamma^2/\gamma_0+8\gamma^2 C(r)\mbf{h}/ \sigma_{b,1} \big)\norm{u - \mathcal{E}(\Pi u)}{}^2
\\
&\qquad
+
2\norm{a \nabla \big( u-\mathcal{E}(\Pi u)\big)\cdot \mbf{n}  \sigma_{\ddd}^{-1/2}}{\Gamma_{\ddd}}^2
+
2C(r) \norm{\sqrt{\sigma_{\ddd}}\big(u-\mathcal{E}(\Pi u)\big)}{\Gamma_{\ddd}}^2 
\\
&\qquad
+
2 s_h^{a,c}(\eta,\eta)+ 2s_h^b(\eta,\eta) .
\end{split}
\]
\end{lemma}

\begin{proof} Galerkin orthogonality
\eqref{galerkin-orthognality}, for any $v_h\in V_h^r$, gives
\begin{equation}\label{galerkin-orthognality2}
\begin{split}
&\int_{\Omega} \Big( a\nabla \big(u-\mathcal{E}(\Pi u)\big) \cdot \nabla \mathcal{E}(v_h)
+
\mbf{b}\cdot\nabla_h \eta \mathcal{E}(v_h)
+
c \big(u - \mathcal{E}(\Pi u)\big)\mathcal{E}(v_h)\Big)\ud x
\\
&\quad
-
\int_{\Gamma_{\ddd}} \!\Big( a \nabla \big( u-\mathcal{E}(\Pi u)\big)\cdot \mbf{n}\, \mathcal{E}(v_h)
+
a \nabla \mathcal{E}(v_h)\cdot \mbf{n}\, \big(u-\mathcal{E}(\Pi u)\big)
-
\sigma_{\ddd} \big(u-\mathcal{E}(\Pi u)\big) \mathcal{E}(v_h)\Big)\ud s
\\ 
&\quad
-
\int_{\Gamma_{\ddd}^-} \big(\mbf{b}\cdot\mbf{n} \big) \eta \mathcal{E}(v_h)\ud s
-
\sum_{T\in\mathcal{T}}\int_{\partial_-T\backslash\partial\Omega} \big(\mbf{b}\cdot\mbf{n} \big) \lfloor \eta\rfloor \mathcal{E}(v_h)\ud s
+
s_h^{a,c}(\eta,v_h)
+
s_h^b(\eta,v_h) \\ 
&=
\int_{\Omega} \Big( a\nabla \mathcal{E}(\xi) \cdot \nabla \mathcal{E}(v_h)
+
\mbf{b}\cdot\nabla_h \xi \mathcal{E}(v_h)
+
c \mathcal{E}(\xi)\mathcal{E}(v_h)\Big)\ud x
\\
&\quad
-
\int_{\Gamma_{\ddd}} \!\Big( a \nabla \mathcal{E}(\xi) \cdot \mbf{n}\, \mathcal{E}(v_h)
+
a \nabla \mathcal{E}(v_h)\cdot \mbf{n}\, \mathcal{E}(\xi)
-
\sigma_{\ddd} \mathcal{E}(\xi) \mathcal{E}(v_h)\Big)\ud s
\\ 
&\quad
-
\int_{\Gamma_{\ddd}^-} \big(\mbf{b}\cdot\mbf{n} \big) \xi \mathcal{E}(v_h)\ud s
-
\sum_{T\in\mathcal{T}}\int_{\partial_-T\backslash\partial\Omega} \big(\mbf{b}\cdot\mbf{n} \big) \lfloor \xi \rfloor \mathcal{E}(v_h)\ud s
+
s_h^{a,c}(\xi,v_h)
+
s_h^b(\xi,v_h) 
\end{split}
\end{equation}
Now examining the terms appearing on the left hand side of
\eqref{galerkin-orthognality2} involving $\mbf{b}$ we see
that
\begin{equation}
\label{galerkin-orthognality-split}
\begin{split}
\int_{\Omega} \mbf{b}\cdot\nabla_h \eta &\mathcal{E}(v_h) \ud x
-
\int_{\Gamma_{\ddd}^-} \big(\mbf{b}\cdot\mbf{n} \big) \eta \mathcal{E}(v_h)\ud s
-
\sum_{T\in\mathcal{T}}\int_{\partial_-T\backslash\partial\Omega}
\big(\mbf{b}\cdot\mbf{n} \big) \lfloor \eta \rfloor
\mathcal{E}(v_h)\ud s\\
&=
\int_{\Omega} \mbf{b}\cdot\nabla_h \eta\big(\mathcal{E}(v_h)- v_h\big) \ud x
-
\int_{\Gamma_{\ddd}^-} \big(\mbf{b}\cdot\mbf{n} \big) \eta \big(\mathcal{E}(v_h)-v_h\big)\ud s \\
&\quad
-
\sum_{T\in\mathcal{T}}\int_{\partial_-T\backslash\partial\Omega} \big(\mbf{b}\cdot\mbf{n} \big) \lfloor \eta \rfloor \big(\mathcal{E}(v_h)-v_h\big)^+\ud s
-
\int_{\Omega} \Big( \mbf{b}\cdot\nabla_h v_h 
+
(\nabla \cdot \mbf{b}) v_h \Big) \eta \ud x 
\\
&\quad
+
\int_{\Gamma_{\dn}^+} \big(\mbf{b}\cdot\mbf{n} \big) \eta v_h\ud s
+
\sum_{T\in\mathcal{T}}\int_{\partial_-T\backslash\partial\Omega} \big(\mbf{b}\cdot\mbf{n} \big) \lfloor v_h \rfloor \eta^-\ud s.
\end{split}
\end{equation}
Recalling that $\mbf b$ satisfies assumption \eqref{convection_field},
by using the orthogonality of $\Pi$, we immediately have
$\int_{\Omega} \mbf{b}\cdot\nabla_h v_h \eta \ud x =0.
$
Using Lemma \ref{lemma_E1} and the Cauchy-Schwarz inequality, we have from \eqref{galerkin-orthognality-split}
\begin{equation}
\label{relation1}
\begin{split}
\int_{\Omega} \mbf{b}\cdot\nabla_h &\eta \mathcal{E}(v_h) \ud x 
-
\int_{\Gamma_{\ddd}^-} \big(\mbf{b}\cdot\mbf{n} \big) \eta \mathcal{E}(v_h)\ud s
-
\sum_{T\in\mathcal{T}}\int_{\partial_-T\backslash\partial\Omega} \big(\mbf{b}\cdot\mbf{n} \big) \lfloor \eta \rfloor \mathcal{E}(v_h)\ud s\\
&\leq
C(r) \norm{\sqrt{\mbf{h}} \mbf{b} \cdot \nabla_h \eta}{} \norm{\jump{v_h}}{\Gamma_{\dint}}
+
C(r) \sqrt{\beta} \norm{\sqrt{|\mbf{b}\cdot\mbf{n}|}\eta }{\Gamma_{\ddd}^-} \norm{\jump{v_h}}{\Gamma_{\dint}} \\
&\qquad
+
C(r)\sqrt{\beta} \norm{\sqrt{|\mbf{b}\cdot\mbf{n}|} \lfloor \eta \rfloor}{\Gamma_{\dint}} \norm{\jump{v_h}}{\Gamma_{\dint}}
+ \norm{\mbf{b}}{W^{1,\infty}(\Omega)} \gamma_0^{-1} \norm{\sqrt{c_0}v_h}{}\norm{\sqrt{c_0}\eta}{}
\\
&\qquad 
+
\norm{\sqrt{|\mbf{b}\cdot\mbf{n}|}\eta}{\Gamma_{\dn}^+}
\norm{\sqrt{|\mbf{b}\cdot\mbf{n}|}v_h}{\Gamma_{\dn}^+}
+
\norm{\sqrt{|\mbf{b}\cdot\mbf{n}|}\eta^-}{\Gamma_{\dint}}
\norm{\sqrt{|\mbf{b}\cdot\mbf{n}|}\lfloor v_h \rfloor}{\Gamma_{\dint}}.
\end{split}
\end{equation}
Now making use of analogous arguments for the reaction term, the
triangle inequality and Lemma \ref{lemma_E1}, we have
\begin{equation}\label{relation2}
\begin{split}
\int_{\Omega} c \big(u - \mathcal{E}(\Pi u)\big)\mathcal{E}(v_h) \ud x 
&\leq
\gamma \norm{u - \mathcal{E}(\Pi u)}{}\Big( \norm{v_h}{} + \norm{\mathcal{E}(v_h) - v_h}{}\Big)
\\
&\kern-1cm\leq
\gamma\gamma_0^{-1/2} \norm{u - \mathcal{E}(\Pi u)}{} \norm{\sqrt{c_0}v_h}{}
+
\gamma C(r)\mbf{h}^{1/2} \norm{u - \mathcal{E}(\Pi u)}{} \norm{\jump{v_h}}{\Gamma_{\dint}}.
\end{split}
\end{equation}
As for the Dirichlet boundary, assumption \eqref{diffusion_tensor} together with an
inverse inequality, result to
\begin{equation}
\label{relation3}
\begin{split}
\int_{\Gamma_{\ddd}} a \nabla \mathcal{E}(v_h)\cdot \mbf{n}\, \big(u-\mathcal{E}(\Pi u)\big)\ud s
&\leq
\norm{\sigma_{\ddd}^{-1/2}a \nabla \mathcal{E}(v_h)\cdot \mbf{n}}{\Gamma_{\ddd}} \norm{\sigma_{\ddd}^{1/2}\big(u-\mathcal{E}(\Pi u)\big)}{\Gamma_{\ddd}}
\\
&\leq
C(r) \norm{\sigma_{\ddd}^{1/2}\big(u-\mathcal{E}(\Pi u)\big)}{\Gamma_{\ddd}}
\norm{\sqrt{a}\nabla \mathcal{E}(v_h)}{}
.
\end{split}
\end{equation}
The result then follows by combining
(\ref{galerkin-orthognality2}), \eqref{galerkin-orthognality-split},
\eqref{relation1}, \eqref{relation2}, \eqref{relation3}, thereby concluding
the proof.
\end{proof}

\begin{theorem}\label{Theorem:error}
Suppose that the assumptions of Theorem \ref{lem:inf-sup} hold and let
$u$ and $u_h$ satisfy the assumptions of Proposition
\ref{prop:gal-orthog}. Suppose further that $u\in
\prod_{T\in\mathcal{T}} H^k(T)$, for some $k\ge 2$. Then, we have
the following a-priori error bound:
\begin{equation}\label{finall_error_bound}
\begin{split}
&\norm{\sqrt{a}\nabla e}{}^2+\norm{\sqrt{\sigma_{\ddd}} e}{\Gamma_{\ddd}}^2 + \norm{u-u_h}{\db}^2
+ \norm{\sqrt{\delta \mbf{h}} (\mbf{b} \cdot \nabla_h (u-u_h))}{}^2 \\ 
&+ s_h^{a,c}(u-u_h,u-u_h)+s_h^{b}(u-u_h,u-u_h) 
\leq C \sum_{T\in {\mathcal{T}}} \Big( 
\mathcal{D}_T + \mathcal{C}_T
\Big) \mbf{h}^{2l-2}|u|_{l,T}^2,
\end{split}
\end{equation}
where 
\begin{equation}\label{error_diffusion}
\mathcal{D}_T = \alpha + \alpha^2 \mbf{h}^{-1} \sigma_{\ddd}^{-1} +h\sigma_{\ddd} ,
\end{equation}
and
\begin{equation}\label{error_convection}
\begin{split}
\mathcal{C}_T &= \beta^2 \mbf{h}/\sigma_{b,1} +\beta \mbf{h}
+ (\gamma^2/\gamma_0) \mbf{h}^2 + ( \norm{\mbf{b}}{W^{1,\infty}(\Omega)}^2\norm{c_0}{L^\infty(\Omega)} /\gamma^2_0) \mbf{h}^2 
+ \norm{c_0}{L^\infty(\Omega)} \mbf{h}^2  \\
&\qquad
+\gamma^2 C(r)\mbf{h}^3/ \sigma_{b,1} + \delta \beta^2 \mbf{h}
+ \sigma_{a,c} \mbf{h} +\sigma_{b,1} \mbf{h} 
+ \sigma_{b,2} \beta^2 \mbf{h},
\end{split}
\end{equation}
for $l=\min\{k,r+1\}$, with $C$ positive constant, independent of $u$ and of $\mbf{h}$.
\end{theorem}

\begin{proof}
Using the notation of the proof of Lemma
\ref{lem:splitty}, triangle inequality implies
\begin{equation}\label{error_split}
\begin{split}
&\norm{\sqrt{a}\nabla e}{}^2+\norm{\sqrt{c_0} (u-u_h)}{}^2+\norm{\sqrt{\sigma_{\ddd}} e}{\Gamma_{\ddd}}^2
+\frac{1}{2}\norm{\sqrt{|\mbf{b}\cdot\mbf{n}|} (u-u_h)}{\Gamma_{\ddd}^-}^2\\
&+\frac{1}{2}\norm{\sqrt{|\mbf{b}\cdot\mbf{n}|} \lfloor u-u_h\rfloor }{\Gamma_{\dint}}^2+\frac{1}{2}\norm{\sqrt{|\mbf{b}\cdot\mbf{n}|} (u-u_h) }{\Gamma_{\dn}^{+}}^2 
+ \norm{\sqrt{\delta \mbf{h}} (\mbf{b} \cdot \nabla_h (u-u_h))}{}^2 \\ 
& + s_h^{a,c}(u-u_h,u-u_h)+s_h^{b}(u-u_h,u-u_h) \\ 
&\leq \Big( 2 \norm{\sqrt{a}\nabla ( u - \mathcal{E}(\Pi u))}{}^2
+
2 \norm{\sqrt{c_0} \eta}{}^2+ 2\norm{\sqrt{\sigma_{\ddd}} (u-\mathcal{E}(\Pi u))}{\Gamma_{\ddd}}^2
\\
&\qquad +
\norm{\sqrt{|\mbf{b}\cdot\mbf{n}|} \eta}{\Gamma_{\ddd}^-}^2
+
\norm{\sqrt{|\mbf{b}\cdot\mbf{n}|} \lfloor \eta \rfloor }{\Gamma_{\dint}}^2
+
\norm{\sqrt{|\mbf{b}\cdot\mbf{n}|} \eta }{\Gamma_{\dn}^{+}}^2 
+
2\norm{\sqrt{\delta \mbf{h}} (\mbf{b} \cdot \nabla_h \eta )}{}^2
\\
&\qquad +
2s_h^{a,c}(\eta,\eta)+2s_h^{b}(\eta,\eta) \Big)
+
\Big( 2 \nsdg{\xi}^2 \Big) =:
\rom{1}+ \rom{2}.
\end{split}
\end{equation}
The optimal approximation properties of $L^2$-orthogonal projection operator, combined with Lemma \eqref{lemma_E1}, yield

\begin{equation}\label{error_part1}
\rom{1}
\leq C \sum_{T\in {\mathcal{T}}} \Big( \alpha + \sigma_{\ddd}\mbf{h}+ \mbf{h}\big(\mbf{h} \norm{c_0}{L^\infty(\Omega)} +\beta 
+ \delta\beta^2  + \sigma_{a,c} +\sigma_{b,1}  
+ \sigma_{b,2} \beta^2 \big)
\Big) \mbf{h}^{2l-2}|u|_{l,T}^2.
\end{equation}
Next, the inf-sup condition given in Theorem
\ref{lem:inf-sup} along with Lemma \ref{lem:splitty} give
\begin{equation}\label{error_part2}
\begin{split}
\rom{2} 
&\leq
\frac{4}{\Lambda^{2}} \Big( 2\norm{\sqrt{a}\nabla \big(u-\mathcal{E}(\Pi u)\big)}{}^2
+
C(r)/\sigma_{b,1} \norm{\sqrt{\mbf{h}} \mbf{b} \cdot \nabla_h \eta}{}^2
+
C(r)\beta /\sigma_{b,1} \norm{\sqrt{|\mbf{b}\cdot\mbf{n}|}\eta }{\Gamma_{\ddd}^-}^2 \\ 
&\quad
+
C(r)\beta/ \sigma_{b,1} \norm{\sqrt{|\mbf{b}\cdot\mbf{n}|} \lfloor \eta \rfloor}{\Gamma_{\dint}}^2
+
\norm{\sqrt{|\mbf{b}\cdot\mbf{n}|}\eta}{\Gamma_{\dn}^+}^2
+
 \norm{\sqrt{|\mbf{b}\cdot\mbf{n}|}\eta^-}{\Gamma_{\dint}}^2 \\ 
&\quad
+
\gamma^2(\gamma_0^{-1}+ C(r)\mbf{h}/ \sigma_{b,1} ) \norm{u - \mathcal{E}(\Pi u)}{}^2
+
\norm{\mbf{b}}{W^{1,\infty}(\Omega)}^2/\gamma^2_0  \norm{\sqrt{c_0}\eta}{}^2
 \\
&\quad
+
\norm{a \nabla \big( u-\mathcal{E}(\Pi u)\big) \sigma_{\ddd}^{-1/2}}{\Gamma_{\ddd}}^2 
+
C(r) \norm{\sqrt{\sigma_{\ddd}}\big(u-\mathcal{E}(\Pi u)\big)}{\Gamma_{\ddd}}^2 
+ s_h^{a,c}(\eta,\eta)+ s_h^b(\eta,\eta) \Big) \\
& \leq
C \sum_{T\in {\mathcal{T}}} \Big( \alpha + \beta^2 \mbf{h}/\sigma_{b,1} +\beta \mbf{h}
+
(\gamma^2/\gamma_0) \mbf{h}^2 
+
 \norm{\mbf{b}}{W^{1,\infty}(\Omega)}^2\norm{c_0}{L^\infty(\Omega)} /\gamma^2_0 \mbf{h}^2
 \\
&\quad
+
\gamma^2 C(r)\mbf{h}^3/ \sigma_{b,1}
+ 
\alpha^2 \mbf{h}^{-1} \sigma_{\ddd}^{-1} 
+ 
\mbf{h} \sigma_{\ddd} 
+
\sigma_{a,c} \mbf{h} +\sigma_{b,1} \mbf{h} 
+
\sigma_{b,2} \beta^2 \mbf{h}
\Big) \mbf{h}^{2l-2}|u|_{l,T}^2.
\end{split}
\end{equation}
The result then follows by combining the bounds
\eqref{error_part1} and \eqref{error_part2}.
\end{proof}

\begin{remark}
We note that the above a-priori error bound for RFEM
\eqref{finall_error_bound} is optimal with respect to the mesh-size $h$. If $\mbf{b} = \mbf{0}$,
the error bound
\eqref{finall_error_bound} reduces to
\[
\norm{\sqrt{a}\nabla ( u - \mathcal{E}(u_h))}{}^2
+
\norm{\sqrt{\sigma_{\ddd}} (u-\mathcal{E}(u_h))}{\Gamma_{\ddd}}^2
+
s_h^{a,0}(u-u_h, u-u_h)
\leq C \mbf{h}^{2l-2} |u|^2_{l}.
\]
The above error bound is $h$-optimal and coincides the error bound from \cite{RFEM} which was shown for standard meshes consisting of simplices. On the other hand, if in the pure hyperbolic case with
diffusion tensor $a=0$, we deduce the error bound
\begin{equation*}
\begin{split}
\norm{u-u_h}{\db}^2
+
\norm{\sqrt{\delta \mbf{h}} (\mbf{b} \cdot \nabla_h (u-u_h))}{}^2 
+
s_h^{0,c}(u-u_h,u-u_h)
+
s_h^{b}(u-u_h,u-u_h) 
\leq
C \mbf{h}^{2l-1} |u|_l^2,
\end{split}
\end{equation*}
which is also $h$-optimal.
\end{remark} 

\section{Numerical experiments}\label{numerics}

We shall now
investigate numerically the asymptotic behaviour of the proposed R-FEM method on
general polygonal meshes. 
We first introduce a sequence of polygonal meshes, indexed by their
element size, together with the simplicial sub-meshes used; see Figure \ref{fig_one}(a) for an example or \ref{mesh_figure} for a refinement of the latter. We point out that the
sub-triangulations used do not introduce any new points in
the interior of the polygonal mesh to keep the
number of degrees of freedom in the triangulated sub-meshes to a
minimum. As expected from the theory, we have numerically observed that increasing the number of degrees of
freedom in the sub-meshes as a proportion of the polytopic meshes does
not increase the order of the method and only improves the accuracy
marginally. All the polygonal meshes are generated by PolyMesher, cf. \cite{polymesher}. Unless clearly stated, the R-FEM solution $u_h$ is computed by \eqref{method} with the following choices of stabilisation parameters
$C_\sigma$ appearing in $\sigma_{\ddd}$, $\sigma_{a,c}$ in
\eqref{stabilisationa}, $\sigma_{b,1}$ and $\sigma_{b,2}$ in
\eqref{stabilisationb}, all with value equal to $10$.

\begin{figure}[!h]
\centering
\subcaptionbox{$\mathcal T$ with $256$
polygons.}{
\includegraphics[scale=0.35]{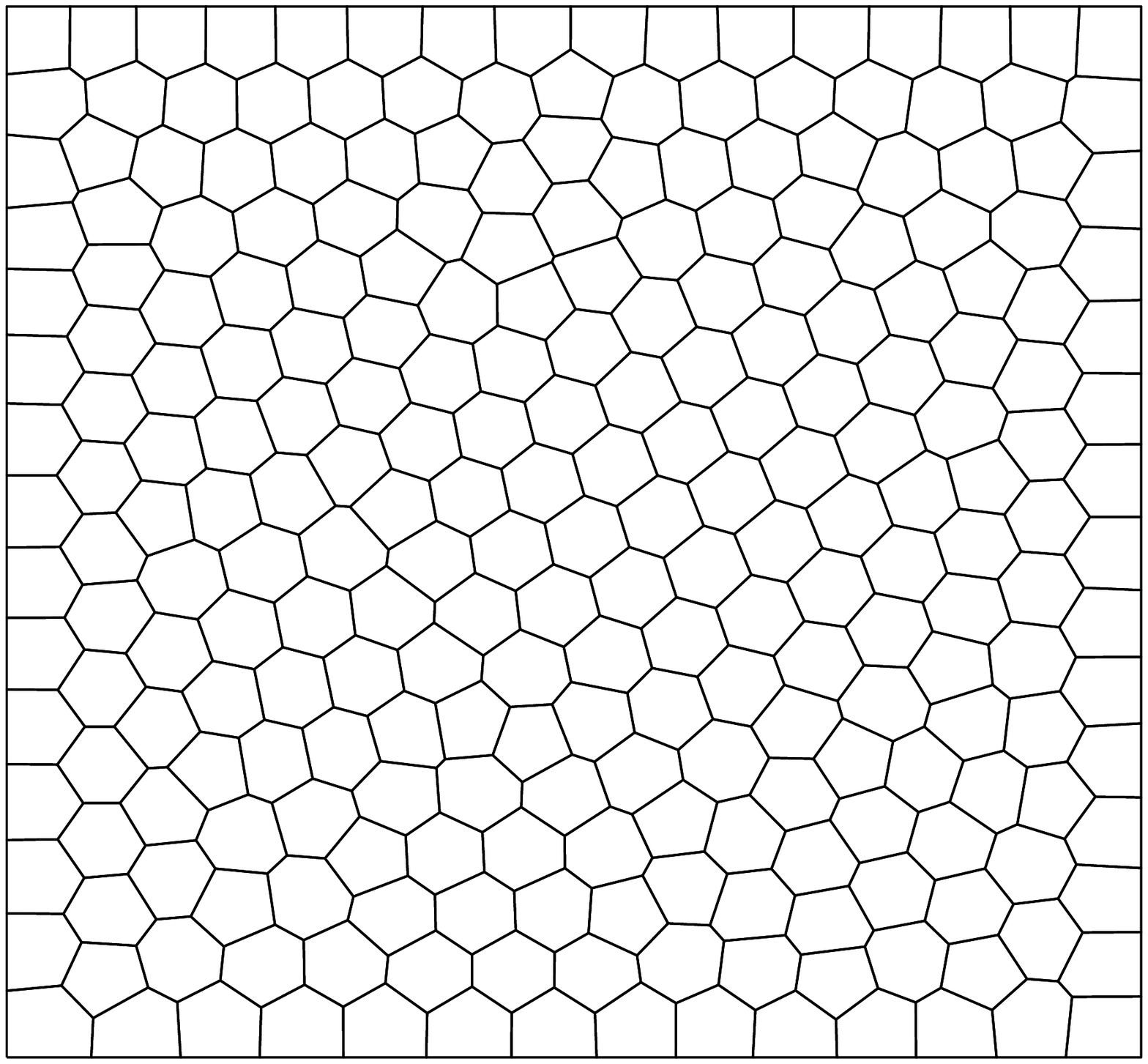}
}
\hfill
\subcaptionbox{$966$-triangle sub-mesh.}{
\includegraphics[scale=0.35]{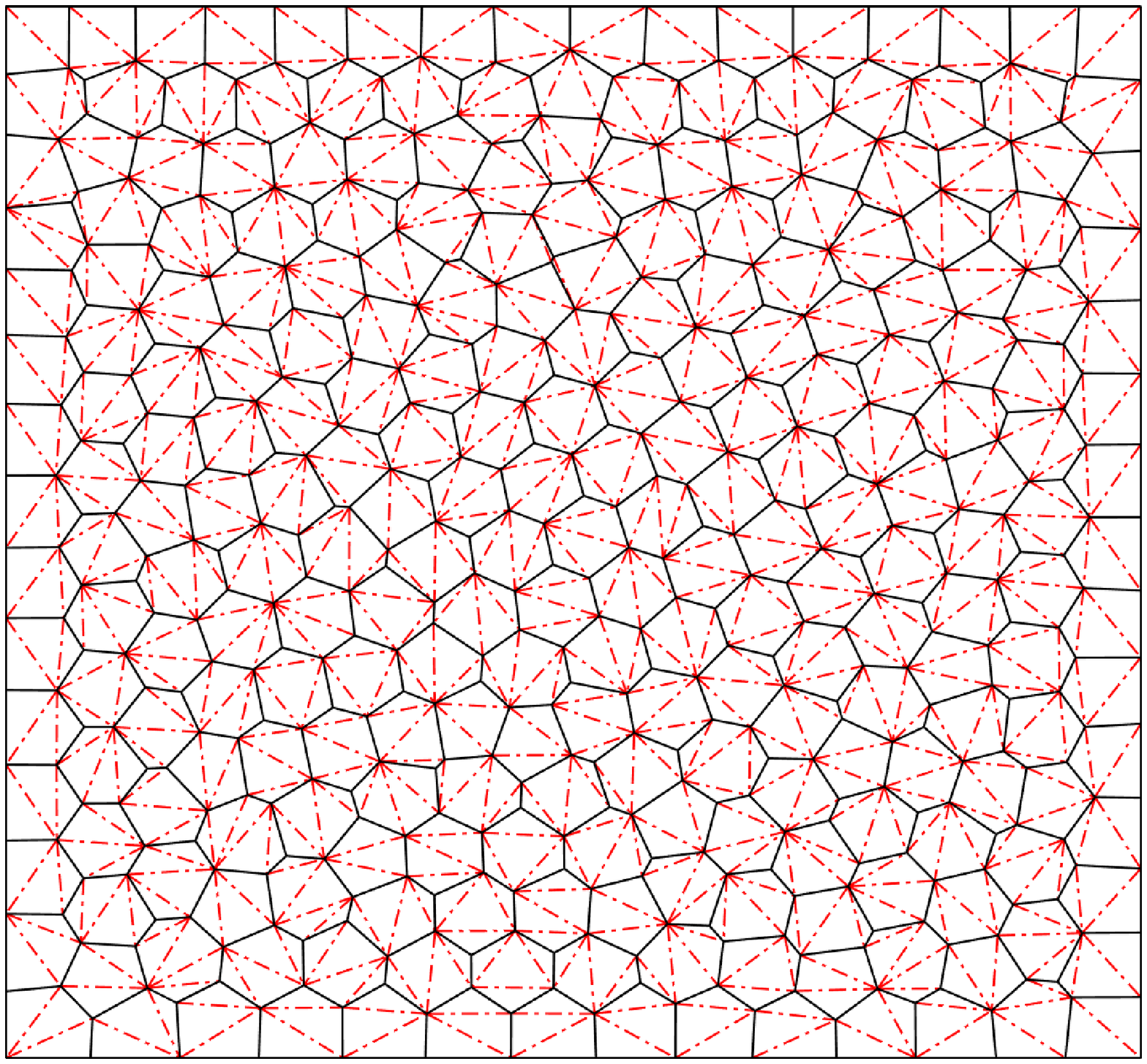} 
}
\caption{
\label{mesh_figure}
An example of a polytopic mesh $\mathcal T$ which is a refinement of the mesh from Figure \ref{fig_one}.}
\end{figure}

\subsection{Example 1: a first order hyperbolic problem}\label{example1}

Let $\Omega:=(0,1)^2$, and choose
\begin{equation}
a \equiv 0, \quad {\bold{b}}=(2-y^2,2-x), \quad c= 1+(1+x)(1+y)^2; 
\end{equation}
the forcing function $f$ is selected so that the analytical solution to \eqref{pde} is given by 
\begin{equation}
u(x,y) = 1+\sin(\pi (1+x)(1+y)^2/8),
\end{equation}
cf. \cite{hss,DGpoly2}. 

\begin{figure}[h!]
\centering
\includegraphics[scale=0.4]{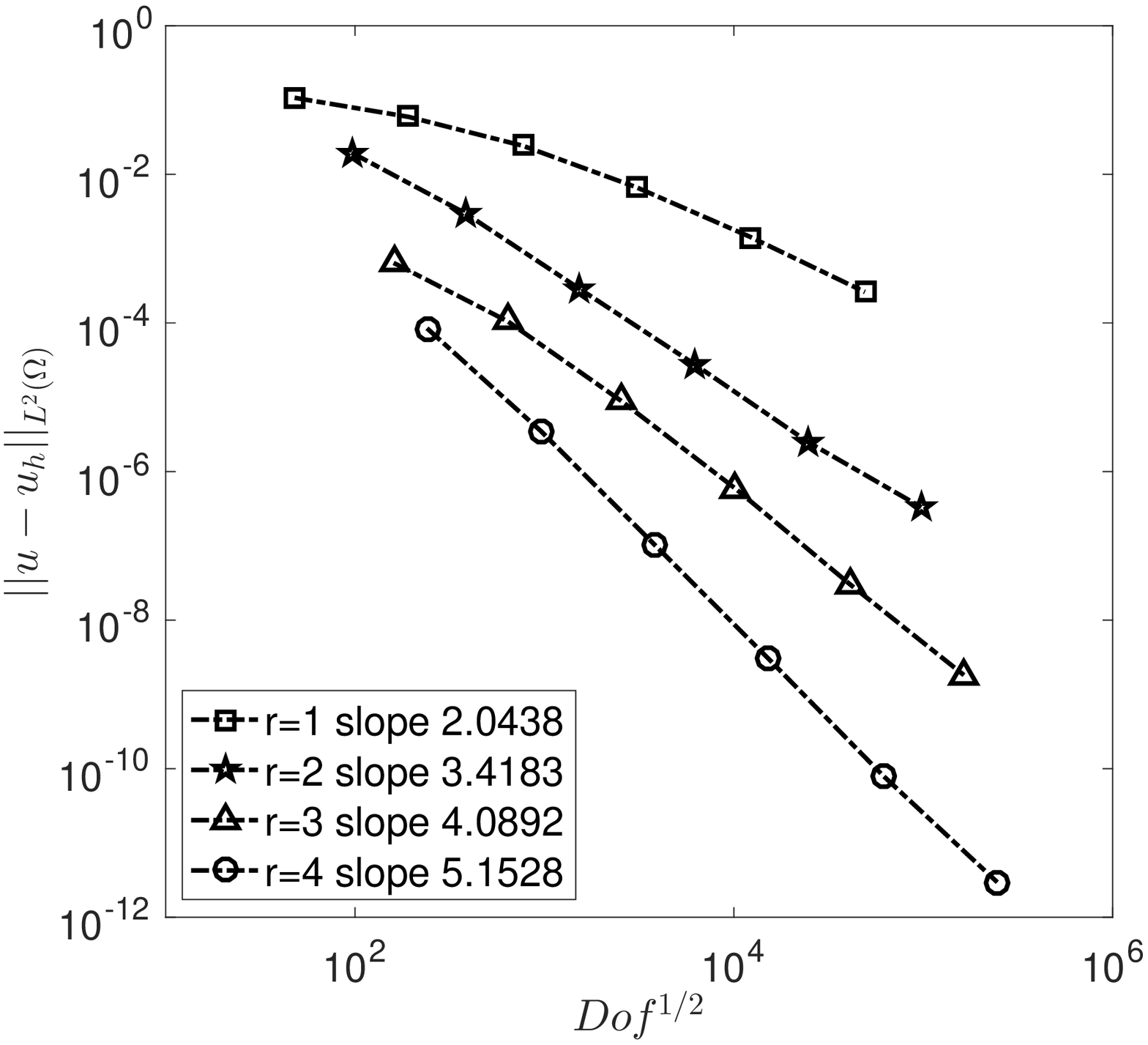}
\hfill
\includegraphics[scale=0.4]{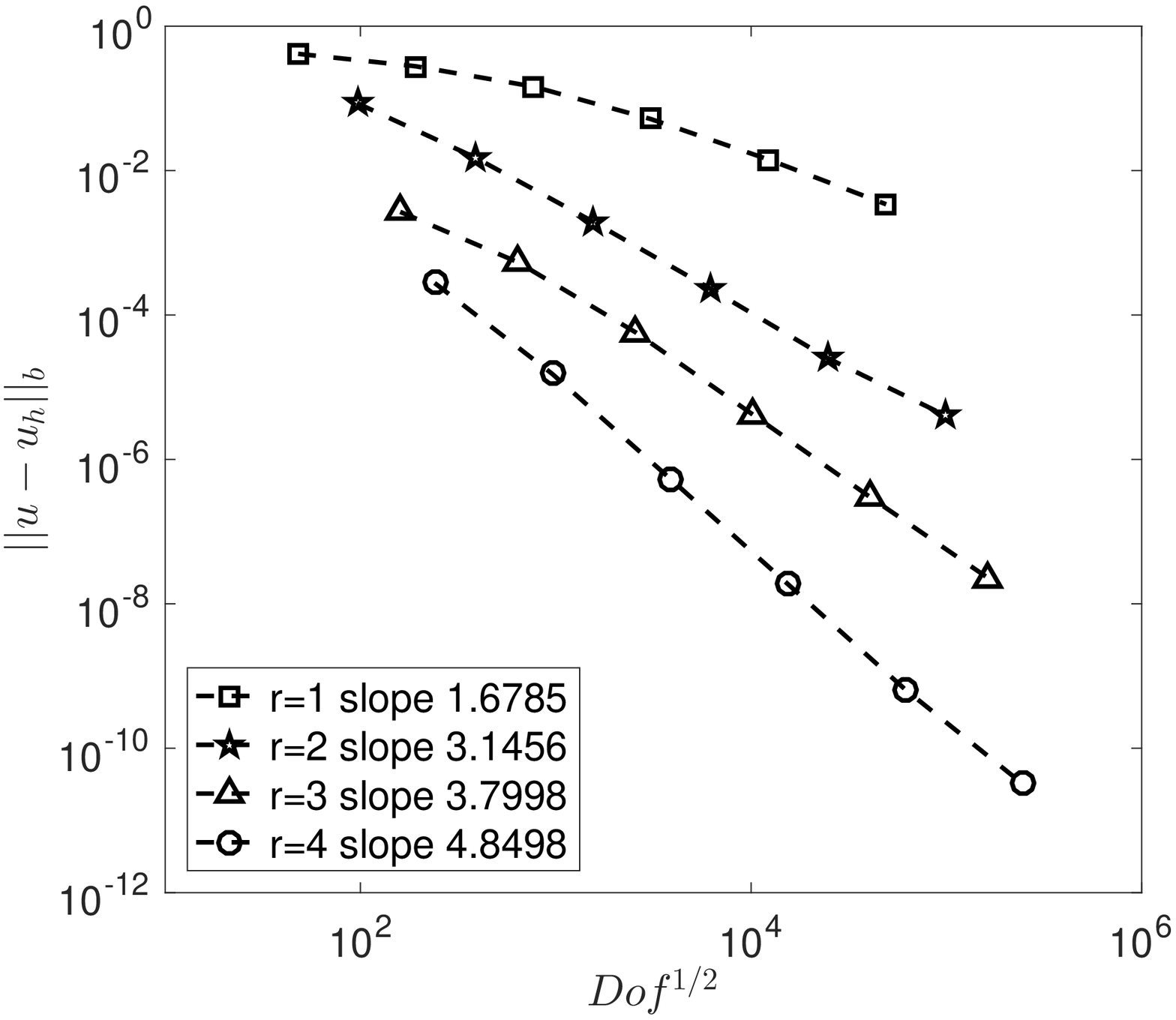}
\caption{
\label{ex1:h-refine}
Example 1. Error against numerical degrees of freedom (Dof). Here
we examine the convergence of the R-FEM under $h$--refinement for
polynomial degrees $r=1,2,3,4$. Notice that
$\norm{u-u_h}{L^2(\Omega)} = O(h^{r+1})$, which is optimal. In
addition, $\norm{u-u_h}{\db}$ appears to converge faster than the analysis of Theorem \ref{Theorem:error} suggests for even polynomial degrees.}
\end{figure}

We examine the convergence behaviour of the R-FEM with respect to
$h$--refinement, with fixed polynomial $r$, for $r=1,\ldots,4$. In
Figure \ref{ex1:h-refine} the error in two
different norms, against the square root of the number of degrees of
freedom in the underlying finite element space $V_h^r$ is given. The slope of
the convergence rate shown is the slope of the last line segment for each convergence line.
We observe that $\|u-u_h \|$ and $\norm{u-u_h}{\db}$ converge to
zero at the optimal rates $\mathcal{O}(h^{r+1})$ and at least
$\mathcal{O}(h^{r+\frac{1}{2}})$, respectively, as the mesh size $h$
tends to zero for each fixed $r$. The latter results agree with the
result \eqref{finall_error_bound} in Theorem \ref{Theorem:error}.
Notice that the behaviour of $\norm{u-u_h}{\db}$ appears to differ as to whether $r$ is odd or even. This is well documented in the study of
discontinuous Galerkin methods for hyperbolic conservation laws
\cite{cockburn2003,gmp2015} and seems to be a feature of the above R-FEM method
as well.

\subsection{Example 2: a nonsymmetric elliptic problem}\label{example2}
Let
$\Omega := (0,1)^2$, and choose
\begin{equation}
a \equiv 1, \quad {\bold{b}}=(1-y,1-x), \quad c\equiv 2; 
\end{equation}
the forcing function $f$ is selected so that the analytical solution
to \eqref{pde} is given by
\begin{equation}
u(x,y) = \sin(\pi x)\sin(\pi y).
\end{equation}
We examine the convergence behaviour of the R-FEM
with respect to $h$-refinement on quasiuniform polygonal meshes, with fixed polynomial $r$, for
$r=1,\dots,4$. In Figure \ref{ex2:h-refine} we plot the error, in
terms of the $L^2$-norm and the (broken) $H^1$-seminorm for both discontinuous R-FEM
approximation $u_h$ and the conforming R-FEM approximation
$\mathcal{E}(u_h)$, against the square root of the number of degrees
of freedom in the underlying finite element space $V_h^r$. We observe that $\|u-u_h \|$ and $|u-u_h|_{H^1{(\Omega,\mathcal{T})}} $
converge to zero at the optimal rates $\mathcal{O}(h^{r+1})$ and
$\mathcal{O}(h^{r})$, respectively, as the mesh size $h$ tends to zero
for each fixed $r$. Moreover, the difference between the R-FEM solutions
$u_h$ and $\mathcal{E}(u_h)$ is marginal. The results are in accordance with
Theorem \ref{Theorem:error}.

\begin{figure}[!ht]
\centering
\includegraphics[scale=0.42]{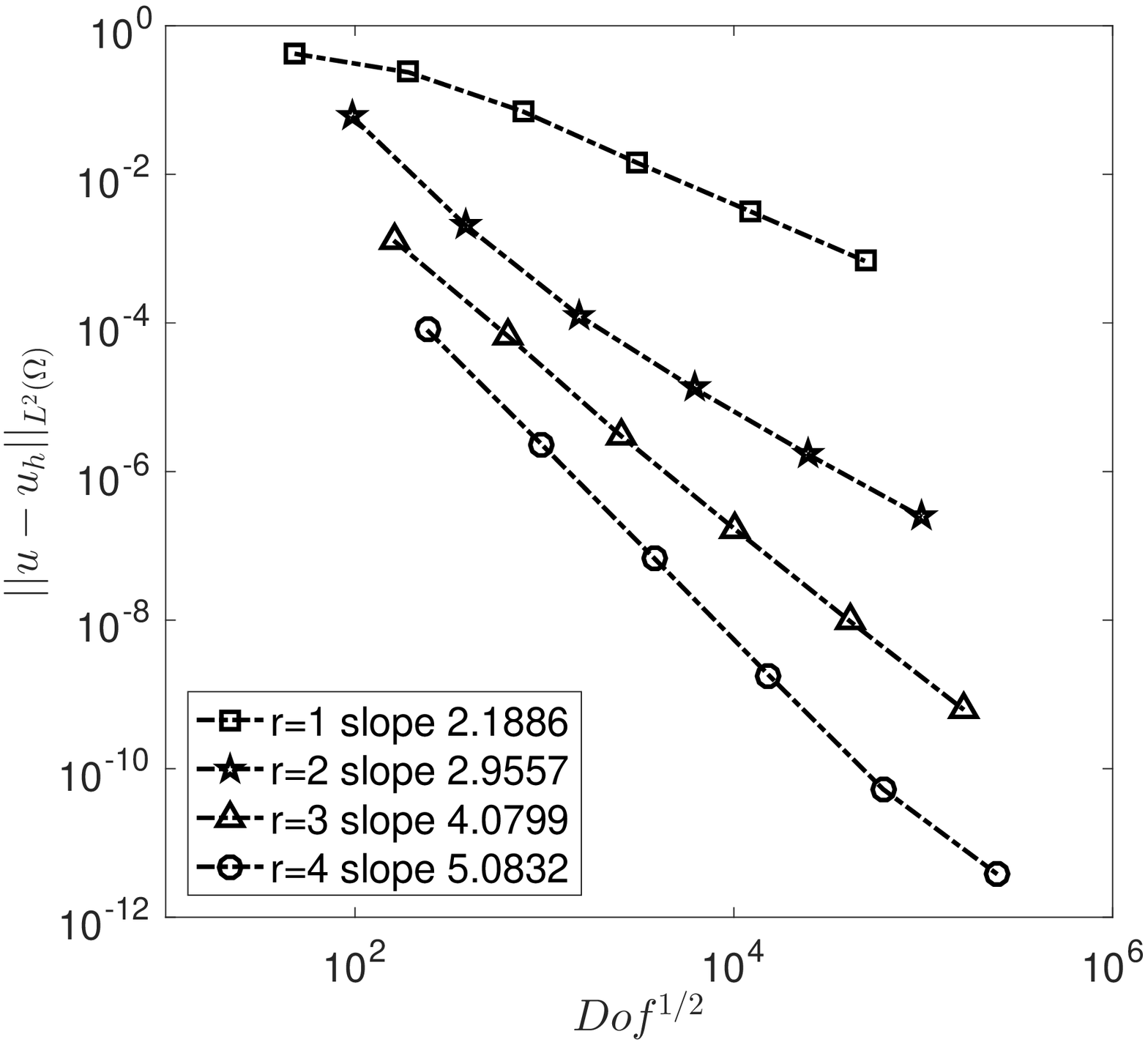} 
\hfill
\includegraphics[scale=0.42]{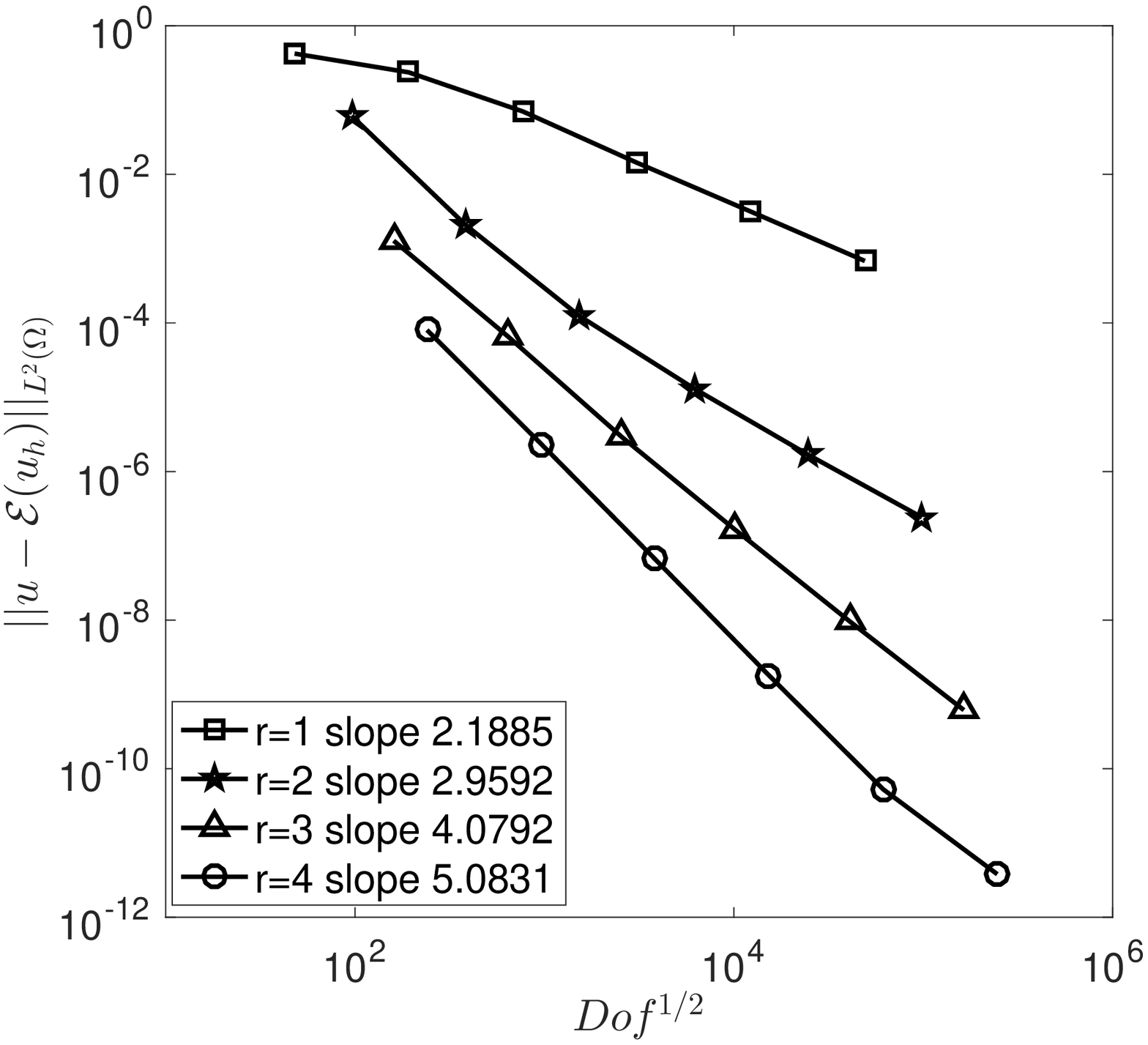} 
\includegraphics[scale=0.42]{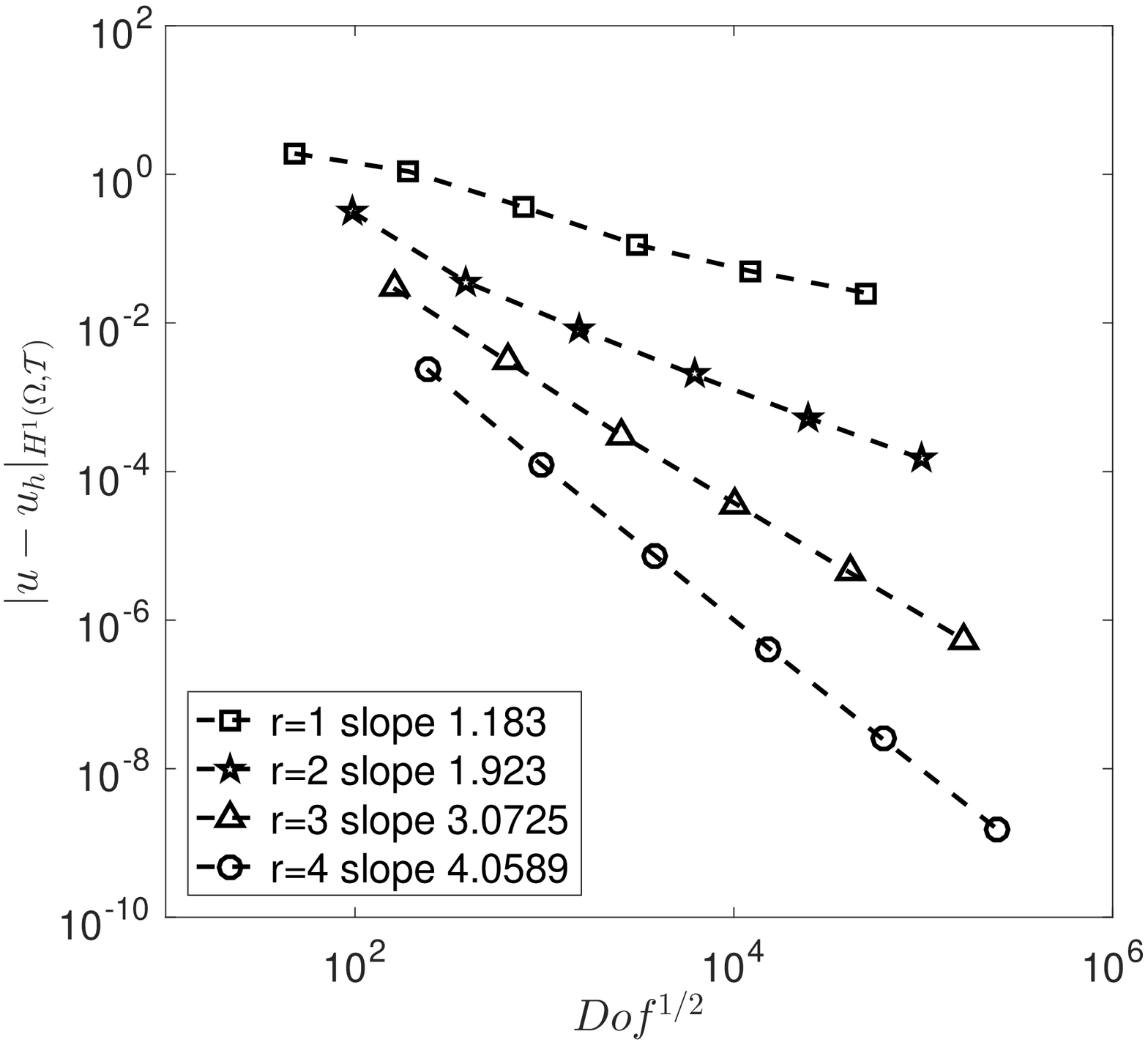}
\hfill
\includegraphics[scale=0.42]{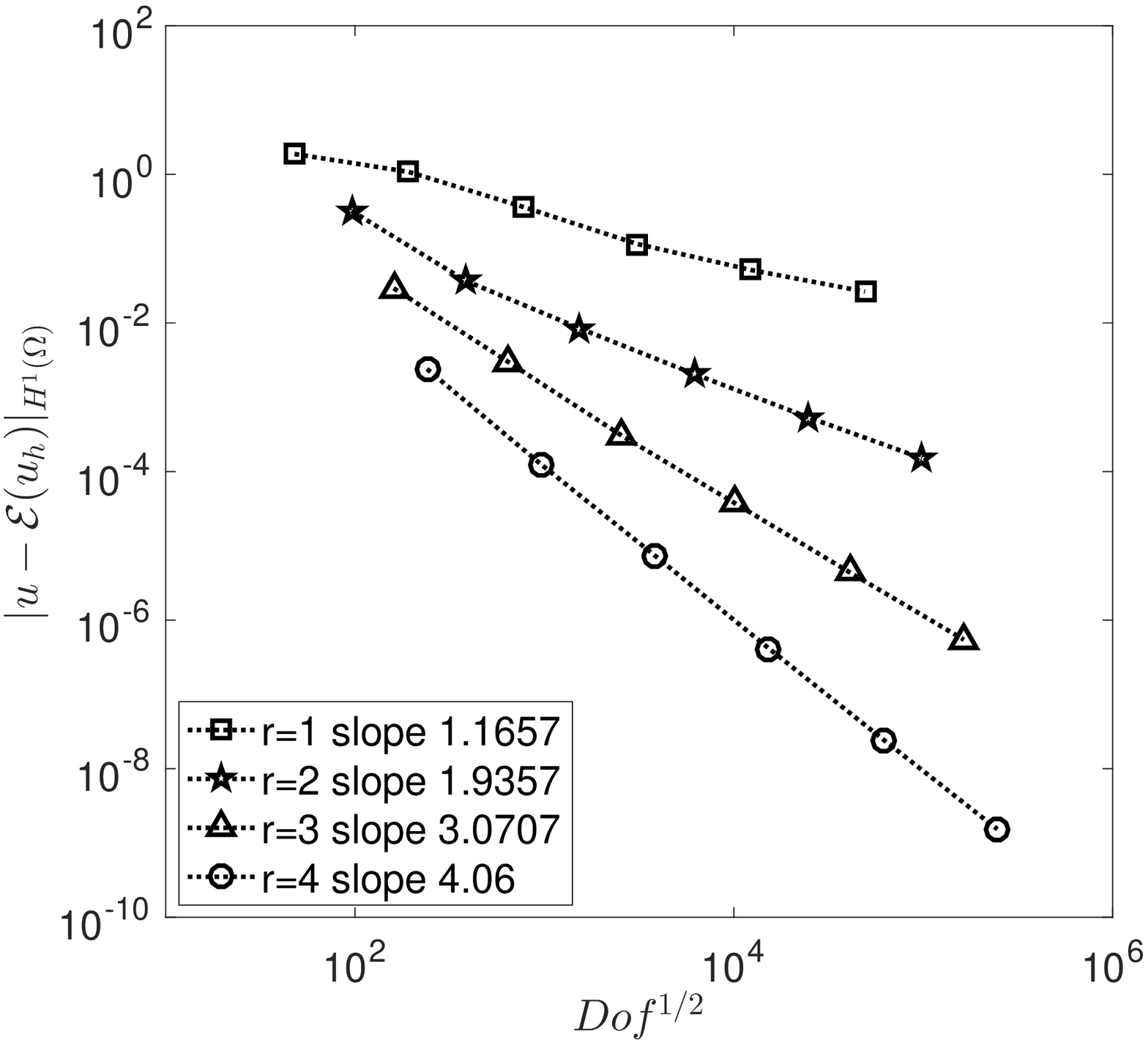}
\caption{
\label{ex2:h-refine}
Example 2. Convergence of the R-FEM under
$h$--refinement for $r=1,2,3,4$. }
\end{figure}

\subsection{Example 3: a advection-dominated elliptic problem}
\label{example3}
We now investigate the numerical stability of the R-FEM
through a series of advection-dominated elliptic problems. Let $\Omega := (0,1)^2$, and choose
\begin{equation}
a \equiv \epsilon, \quad {\bold{b}}=(1,1), \quad c\equiv 0; 
\end{equation}
together with a forcing function $f=1$ and homogeneous Dirichlet
boundary conditions. This example is known to admit boundary
layers in the vicinity of the right and top boundaries $x=1$ and $y=1$ when $\epsilon\ll 1$. We
investigate the stability of R-FEM as $\epsilon\to 0$ on a
fixed, relatively coarse mesh which is insufficient to resolve the
layer. More specifically, we consider a fixed mesh consisting of $1024$
polygons over the domain, we choose $r=1$ and take $\epsilon =10^{-2}, 10^{-4}, 10^{-6}$. In Figure
\ref{ex3:stability} we plot the numerical solutions $u_h$ and
$\mathcal{E}(u_h)$. We observe that for $\epsilon =10^{-2}$, the mesh
is fine enough to resolve the layer and, hence, both $u_h$ and
$\mathcal{E}(u_h)$ are stable. For $\epsilon =10^{-4}$, the mesh is no
longer fine enough to resolve the layer. However, the solutions are
still stable in the sense that neither solution admits non-physical
oscillations near the boundary. In the case $\epsilon =10^{-6}$,
the mesh is too coarse to resolve the layer. Both
$u_h$ and $\mathcal{E}(u_h)$ appear to be stable in this case. Moreover,
the solutions are very close to the solution for the pure hyperbolic
problem with inflow boundary satisfying Dirichlet boundary conditions. This is expected as the boundary conditions have been imposed in a weak fashion for the numerical method to be valid in the hyperbolic limit $\epsilon=0$ also, in the spirit of the classical discontinuous Galerkin methods.

\begin{figure}[!ht]
\centering
\subcaptionbox{$u_h$ with $\epsilon=10^{-2}$}{
\includegraphics[scale=0.33]{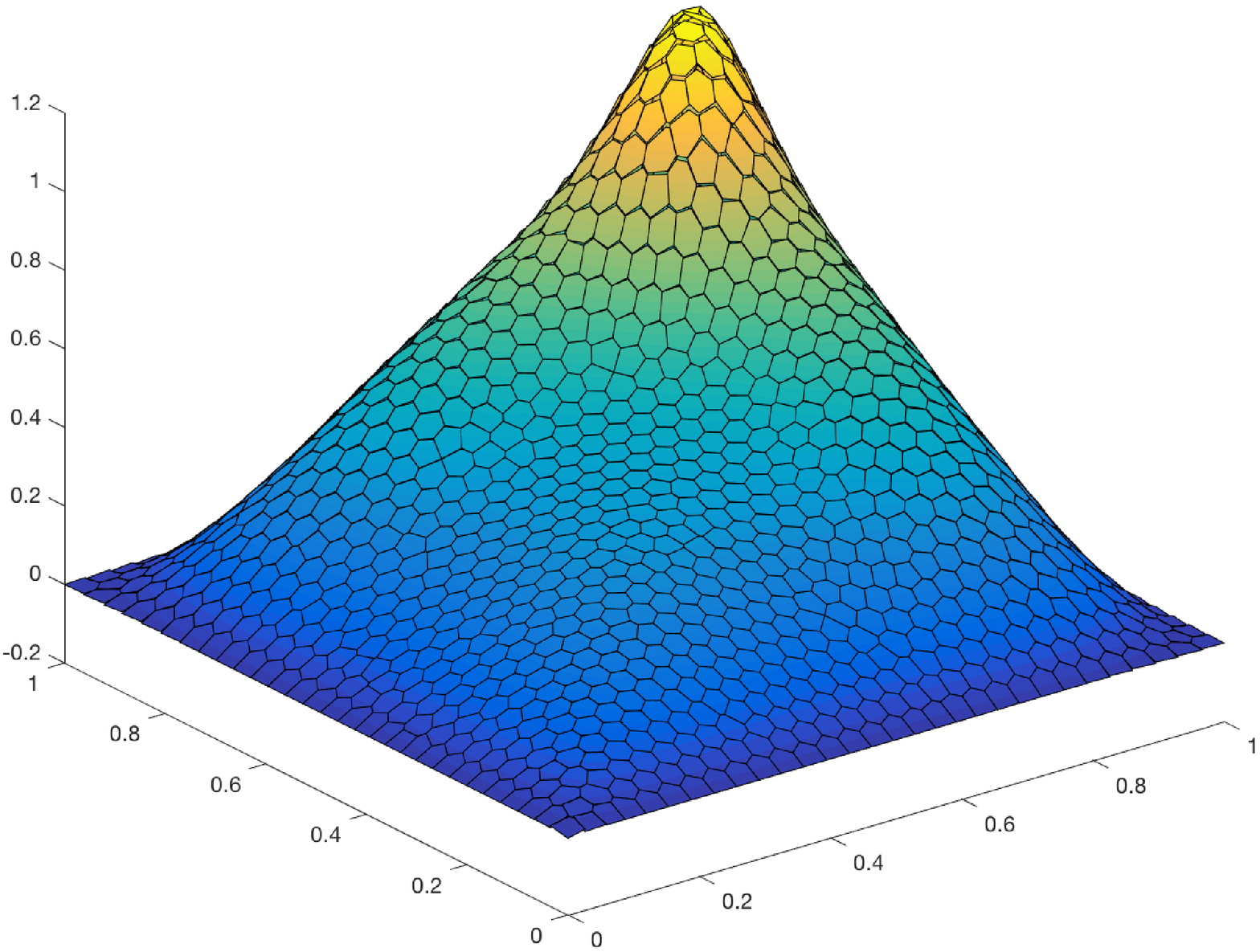} 
}
\subcaptionbox{$\mathcal E(u_h)$ with $\epsilon=10^{-2}$}{
\includegraphics[scale=0.33]{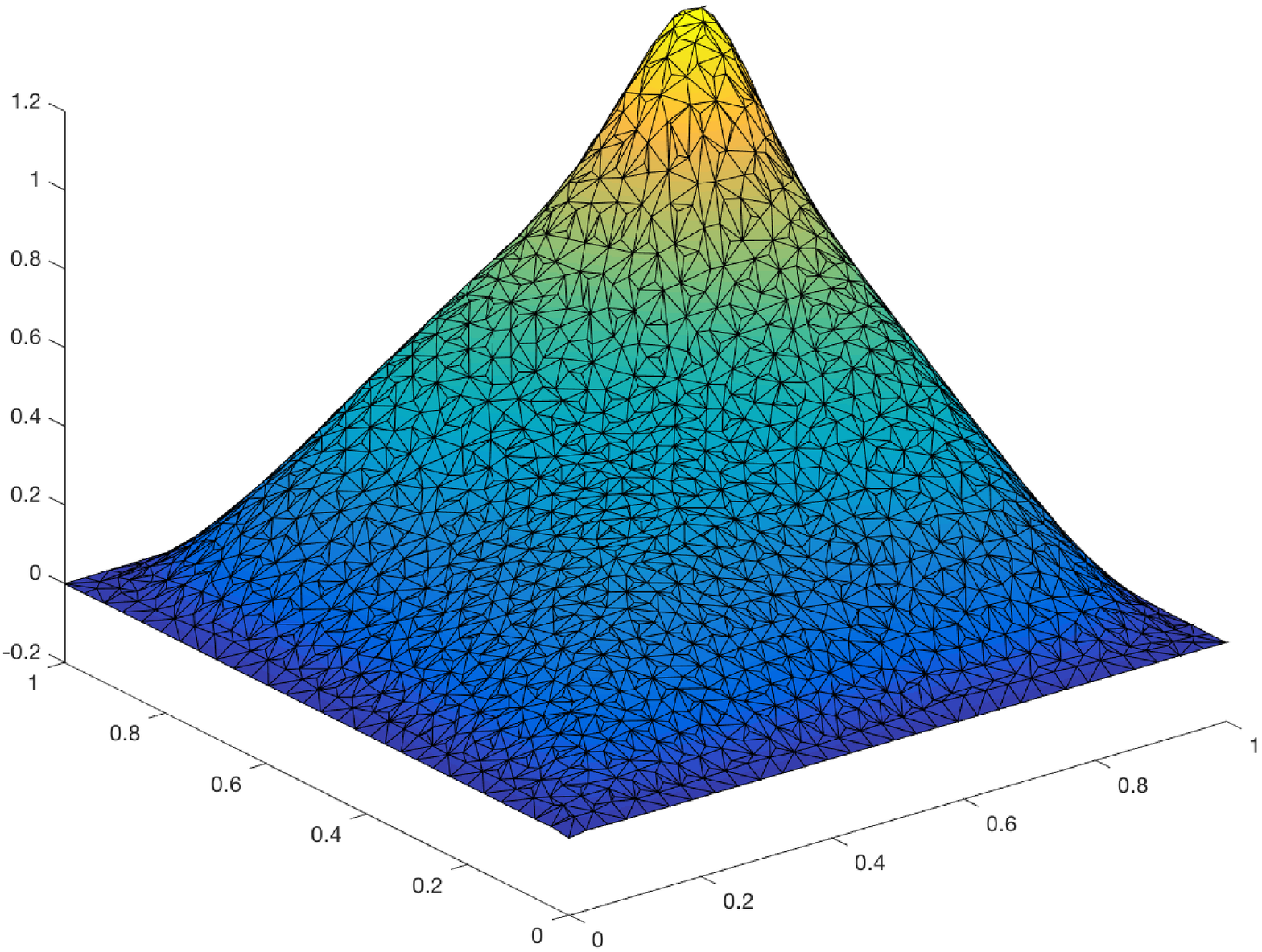}
}
\subcaptionbox{$u_h$ with $\epsilon=10^{-4}$}{
\includegraphics[scale=0.33]{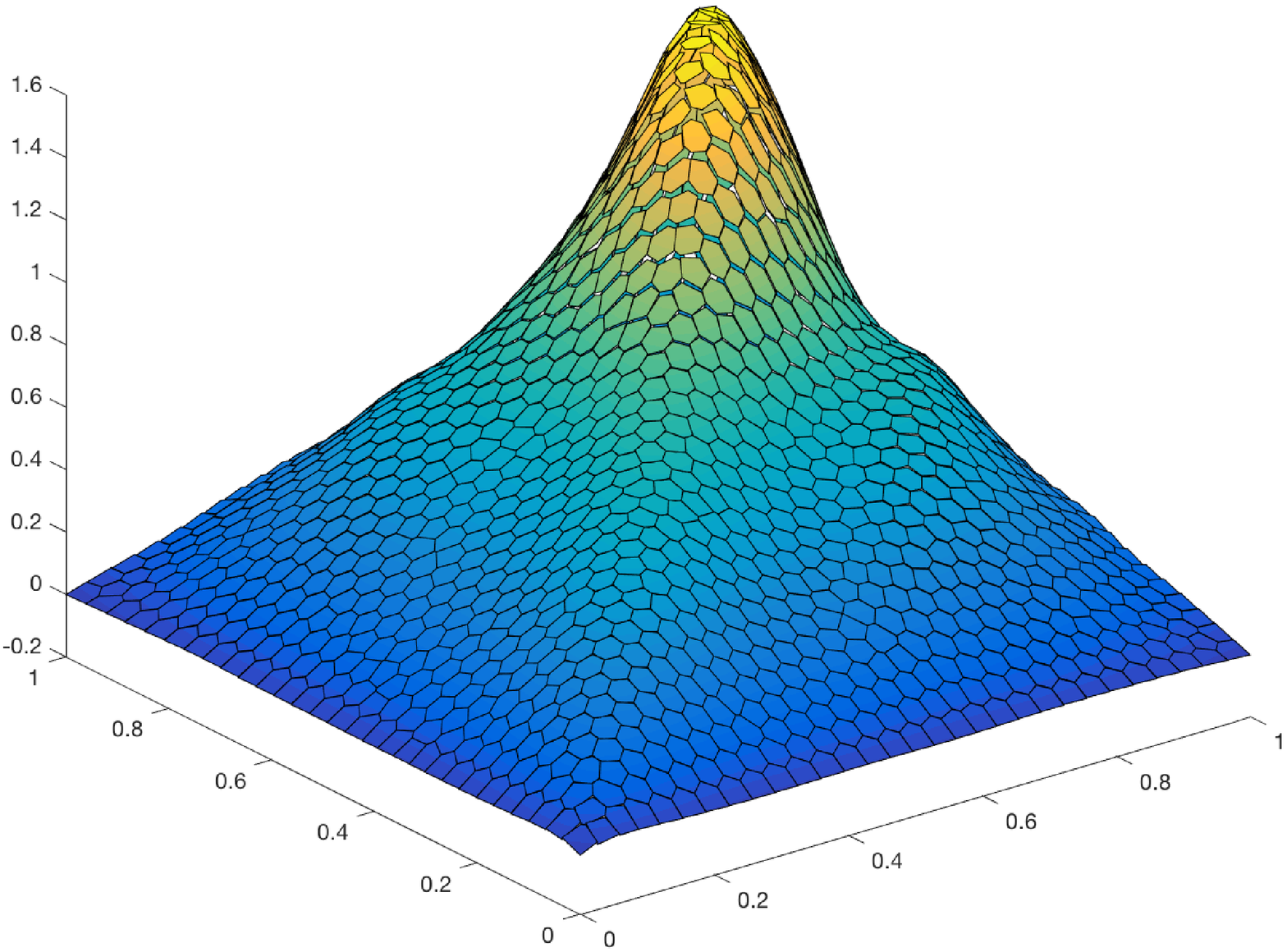} 
}
\subcaptionbox{$\mathcal E(u_h)$ with $\epsilon=10^{-4}$}{
\includegraphics[scale=0.33]{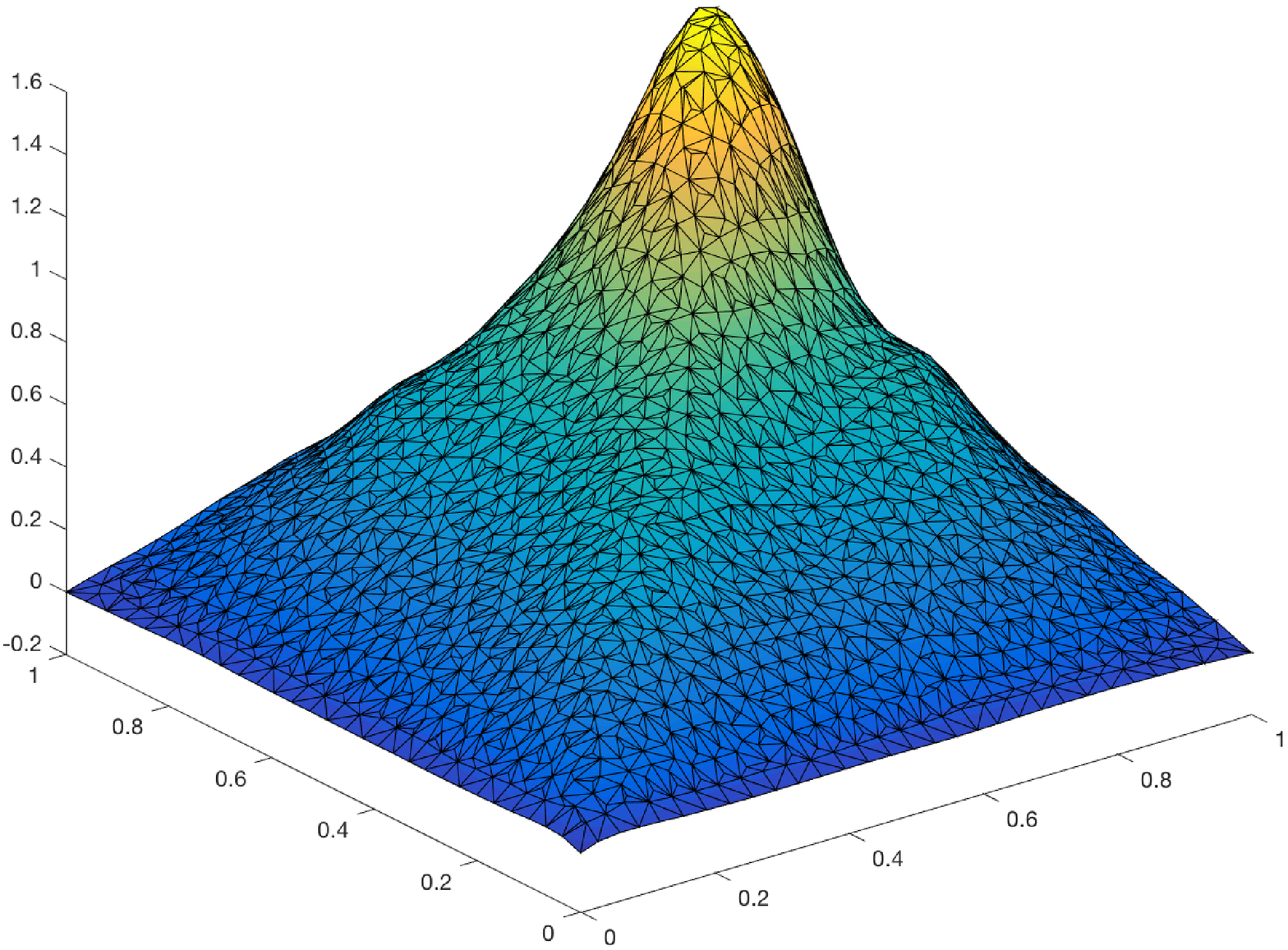} 
}
\subcaptionbox{$u_h$ with $\epsilon=10^{-6}$}{
\includegraphics[scale=0.33]{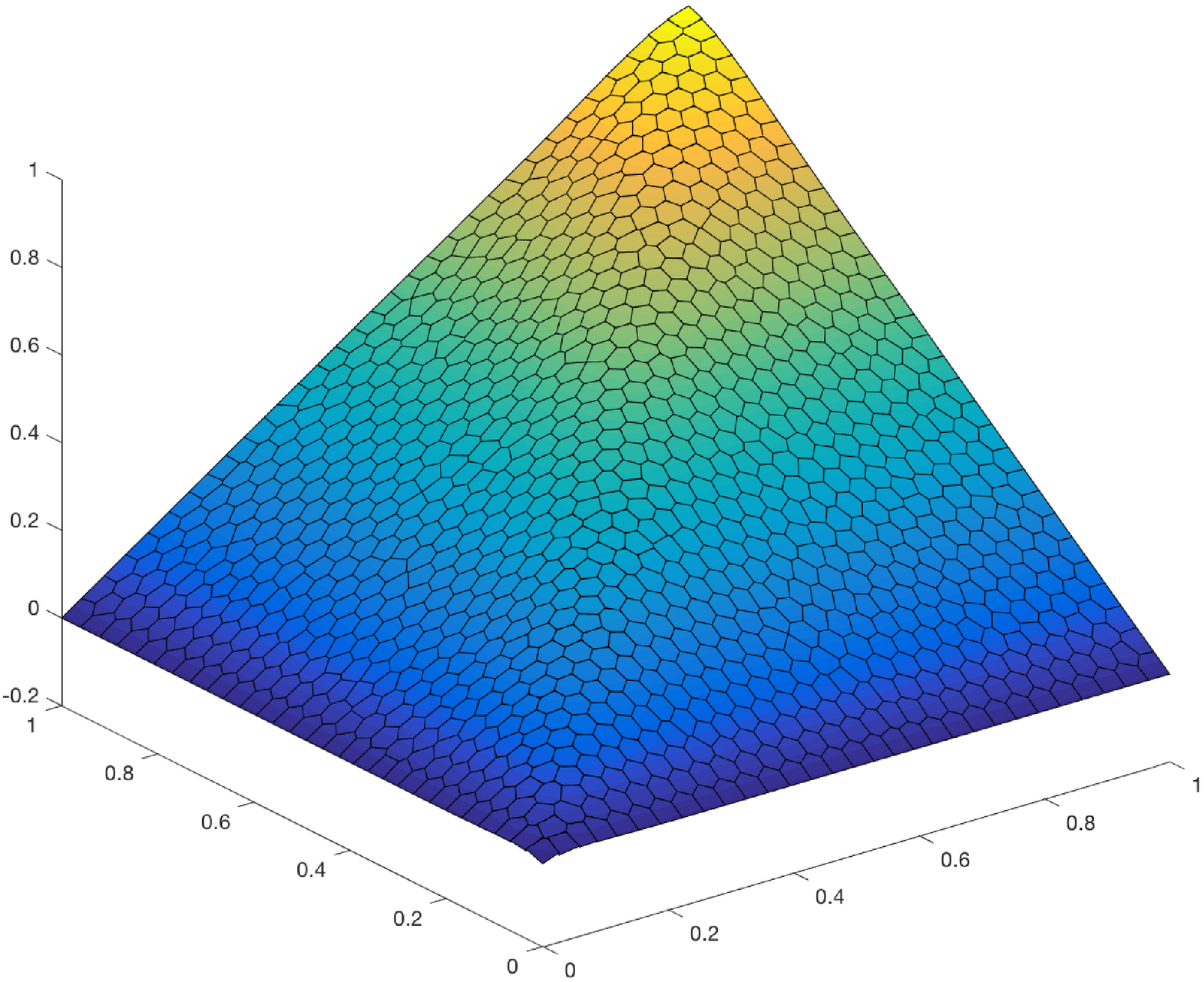}
}
\subcaptionbox{$\mathcal E(u_h)$ with $\epsilon=10^{-6}$}{
\includegraphics[scale=0.33]{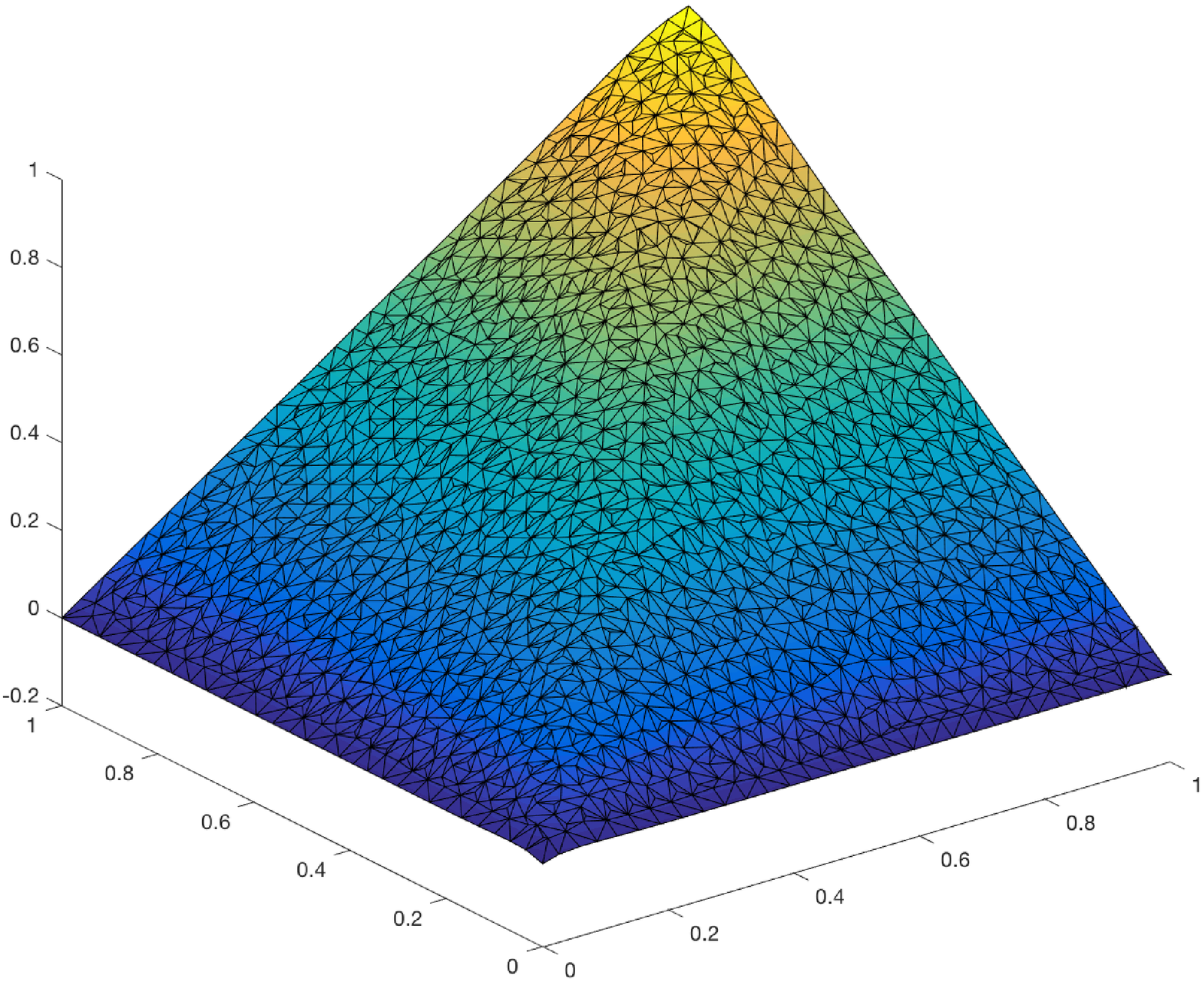}
}
\caption{
\label{ex3:stability}
Example 3. R-FEM solutions for a mesh consisting of $1024$ polygonal
elements and $r=1$.
}
\end{figure}

\subsection{Example 4 -- a mixed-type problem}
\label{example4}

We now consider a partial differential equation with
nonnegative characteristic form of mixed type. To this end, we let
$\Omega=(-1,1)^2$, and consider the PDE problem:
\begin{equation}
\begin{cases}
-x^2u_{yy} + u_x+u = 0, &\quad \text{for } -1 \leq x \leq 1, y > 0 ,\\
u_x+u = 0, &\quad \text{for } -1 \leq x \leq 1, y \leq 0 ,
\end{cases}
\end{equation}
with analytical solution:
\begin{equation}
u(x,y) =
\begin{cases}
\sin (\frac{1}{2}\pi (1+y)) \exp(-(x+\frac{\pi^2x^3 }{12})),& \text{for } -1 \leq x \leq 1, y > 0 ,\\
\sin (\frac{1}{2}\pi (1+y)) \exp(-x), &\text{for } -1 \leq x \leq 1, y\leq 0,
\end{cases}
\end{equation}
cf. \cite{DGpoly2}. This problem is hyperbolic in the region $y\leq
0$ and parabolic for $y > 0$. Notice that, in order to ensure
continuity of the normal flux across $y=0$ where the partial
differential equation changes type, the analytical solution has a
discontinuity across the line $y=0$, cf. \cite{hss}.

\begin{figure}[h!]
\centering
\subcaptionbox{An aligned mesh with $64$
polygons.}{
\includegraphics[scale=0.37]{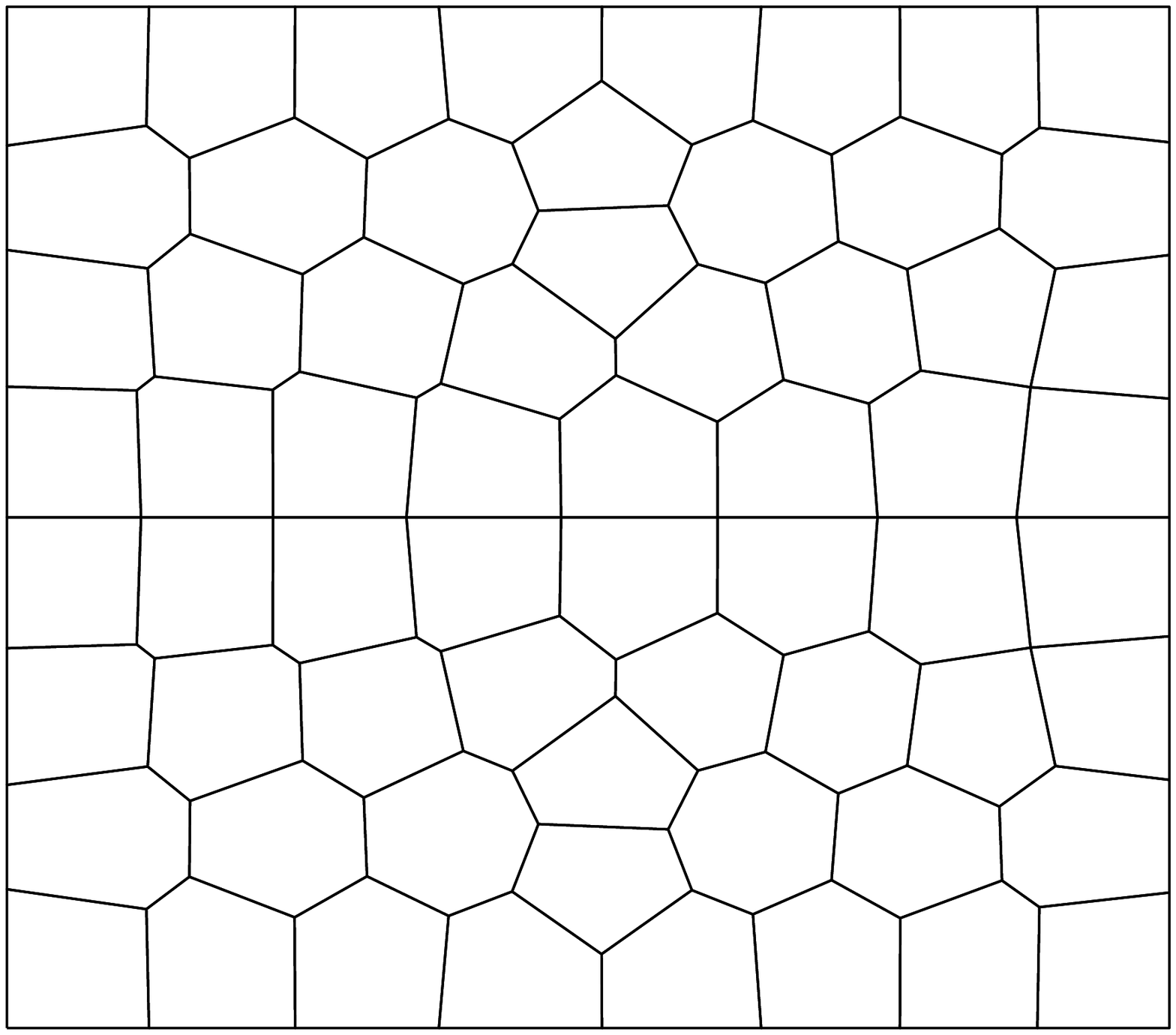}
}
\hfill
\subcaptionbox{A $208$ triangle sub-mesh.}{
\includegraphics[scale=0.38]{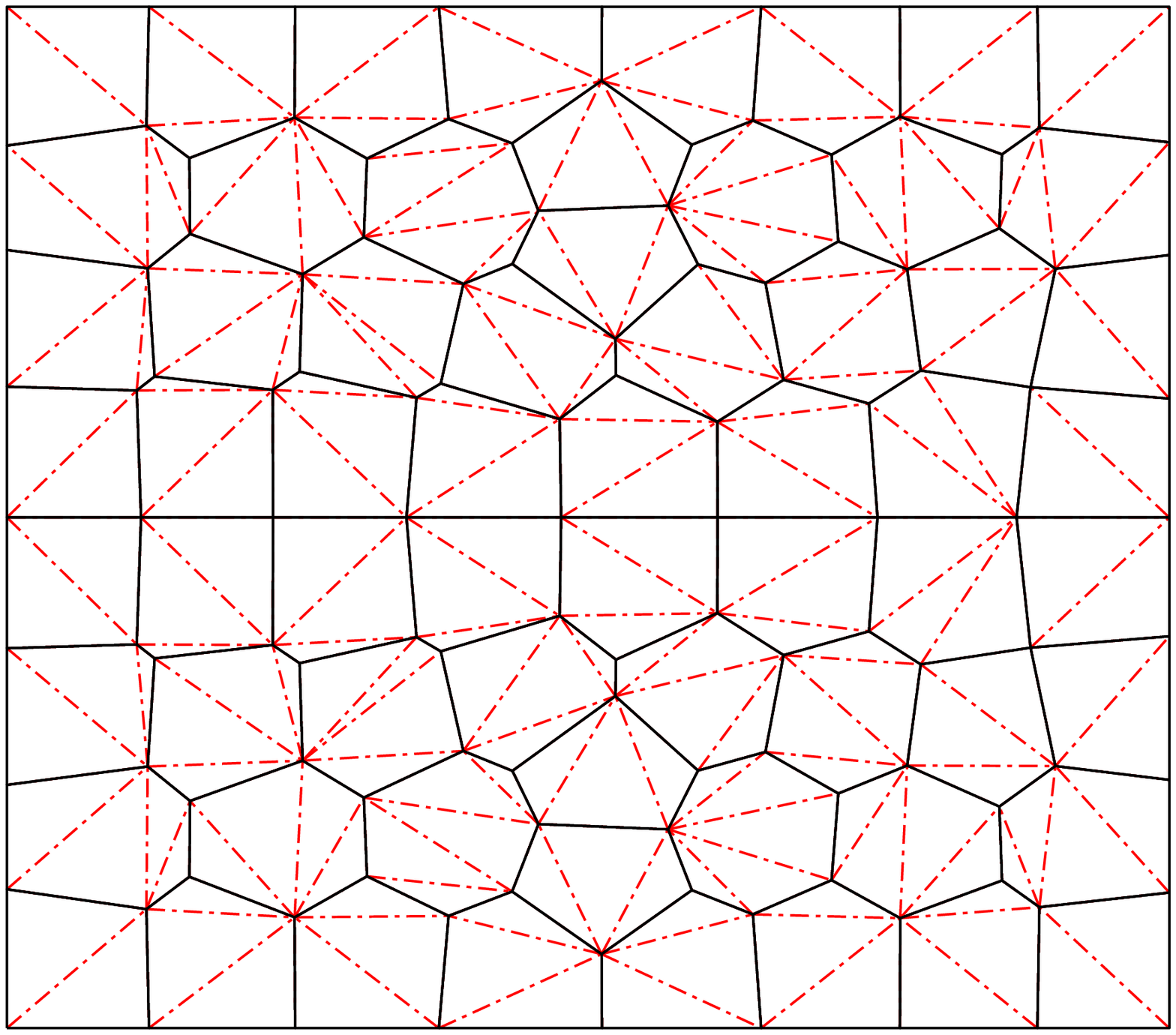}
}
\caption{
\label{aligned_mesh_figure}
An aligned polytopic mesh, $\mathcal T$,
with $64$ polygons and $208$ triangle sub-mesh
$\tilde{\mathcal{T}}$. }
\end{figure}

We examine the convergence behaviour of the R-FEM with respect to
$h$-refinement, with fixed polynomial $r$, for $r=1,\dots,4$. To have the opportunity to possibly observe optimal convergence rates, we align the polygonal mesh with the solution's discontinuity; a typical mesh is shown in
Figure \ref{aligned_mesh_figure}. Also, for this example the recovery operator
is constructed in piecewise fashion over the two
subdomains. This ensures the conforming R-FEM solution
$\mathcal{E}(u_h)$ is able to have a jump discontinuity over the
interface where the problem changes type. 

In Figure \ref{ex4:h-refine} we plot the $L^2$-norm error, as well as the error in  the norm on the left hand-side of \eqref{finall_error_bound} for R-FEM
approximation $u_h$, against the square root of the number of degrees
of freedom in the underlying finite element space $V_h^r$. Abusing the notation, we use again $\nsdg{\cdot}$ to denote the norm on the left hand-side of \eqref{finall_error_bound}. We observe that $\nsdg{u-u_h} $
converges to zero at the optimal rates
$\mathcal{O}(h^{r})$, as the mesh size $h$ tends to zero
for each fixed $r$. These results agree with the
result \eqref{finall_error_bound} in Theorem \ref{Theorem:error}. However, the convergence rate for $\|u-u_h\|$ seems to be slightly suboptimal in $h$. Additionally, we also plot the error in terms of $L^2$-norm and $H^1$-seminorm for $\mathcal{E}(u_h)$, against the square root of the number of degrees
of freedom in the underlying finite element space $V_h^r$. Here, again, observe that $\|u-\mathcal{E} (u_h) \|$ and $|u-\mathcal{E} (u_h) |_{H^1{\Omega}}$
converge to zero at a slightly suboptimal rate.

\begin{figure}[!ht]
\centering
\includegraphics[scale=0.4]{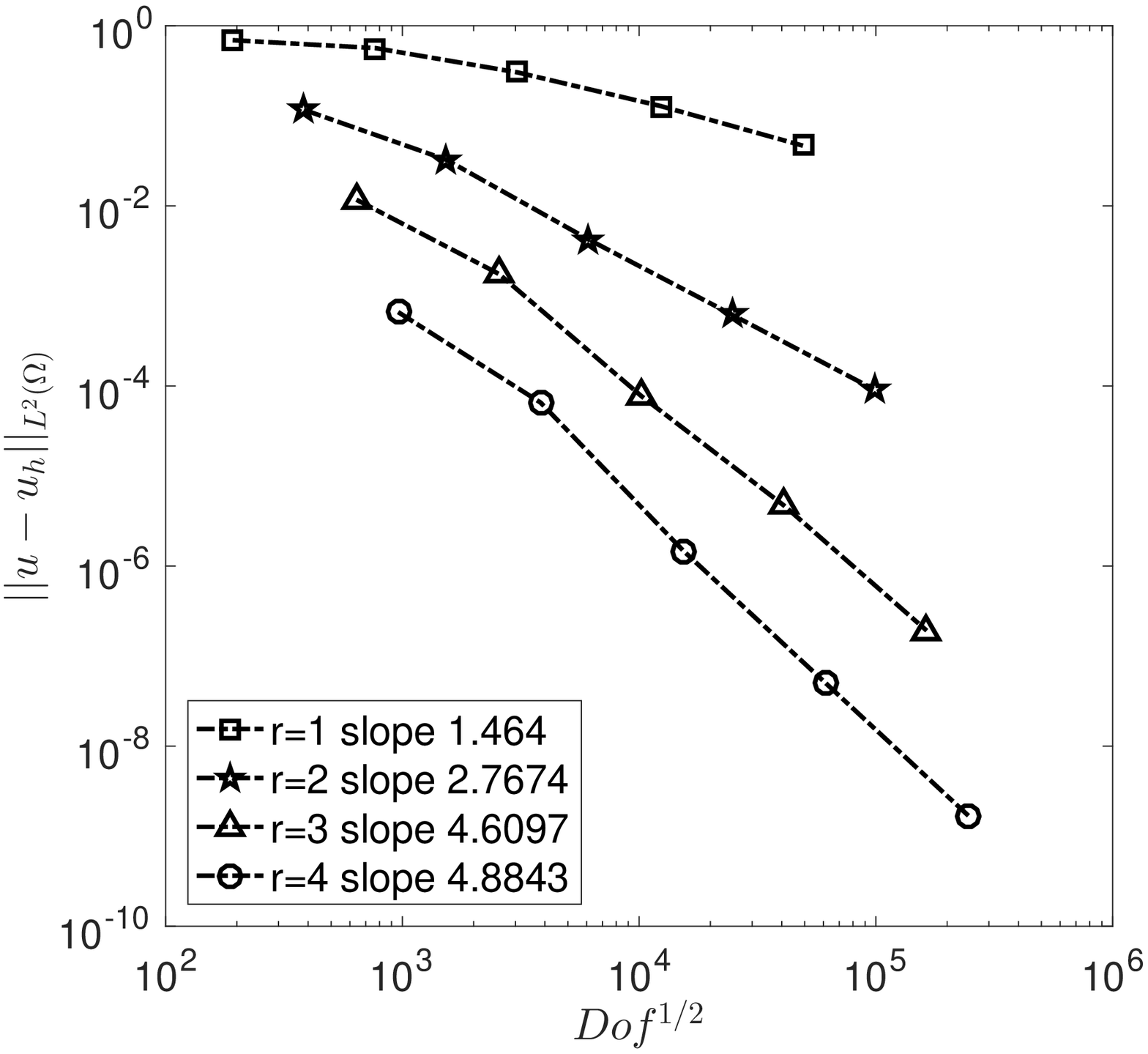}
\hfill
\includegraphics[scale=0.4]{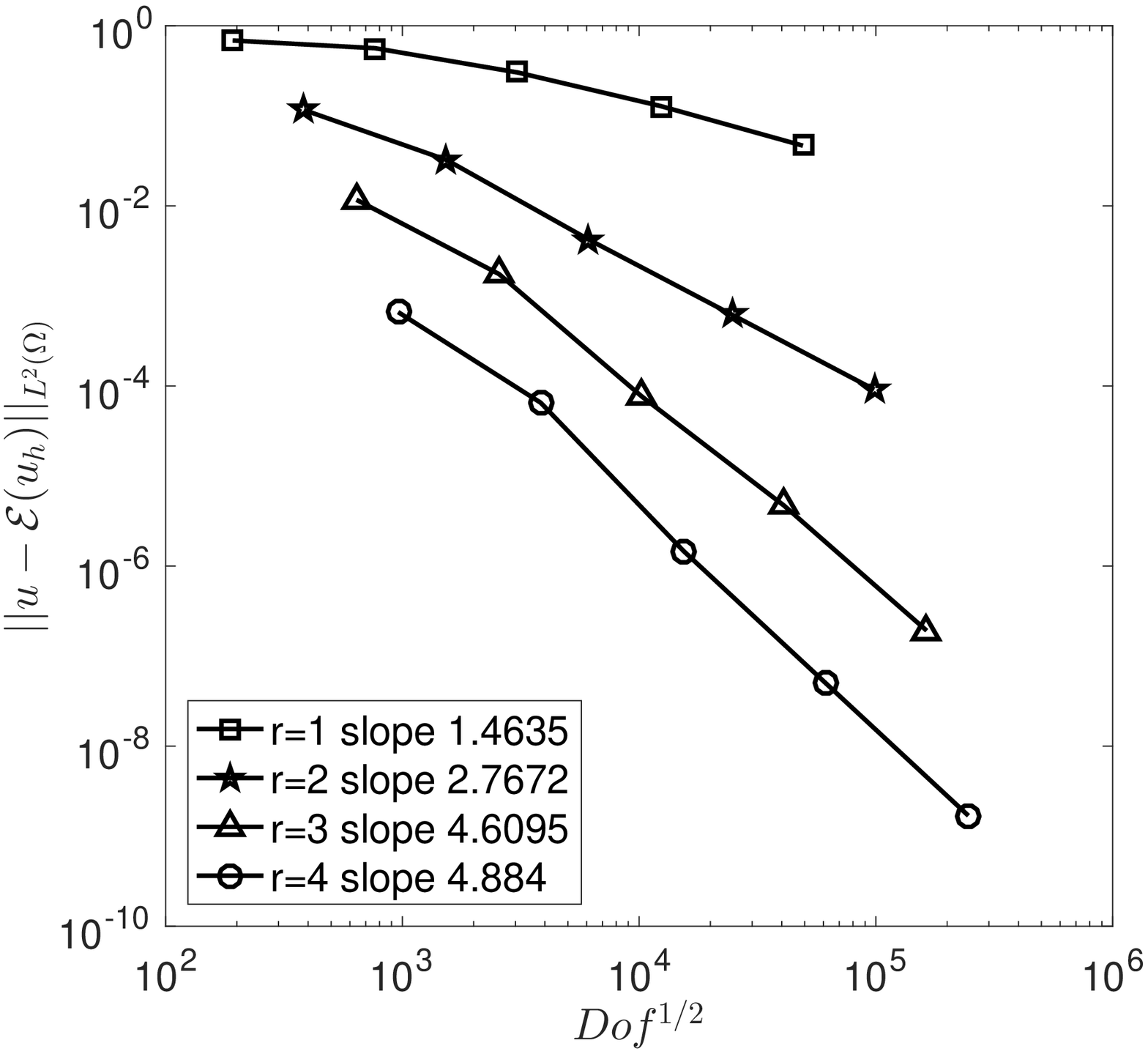}
\includegraphics[scale=0.4]{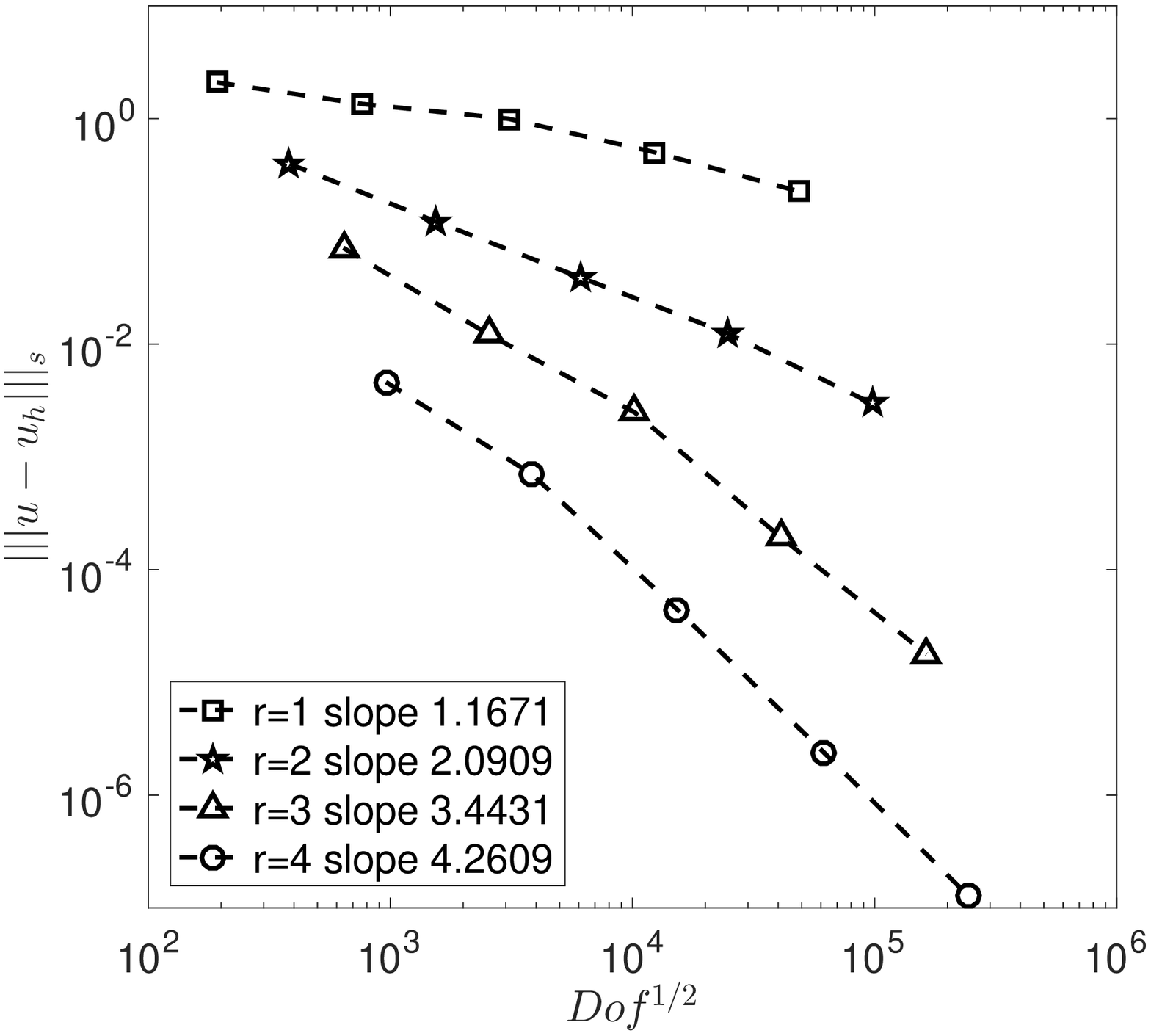}
\hfill
\includegraphics[scale=0.4]{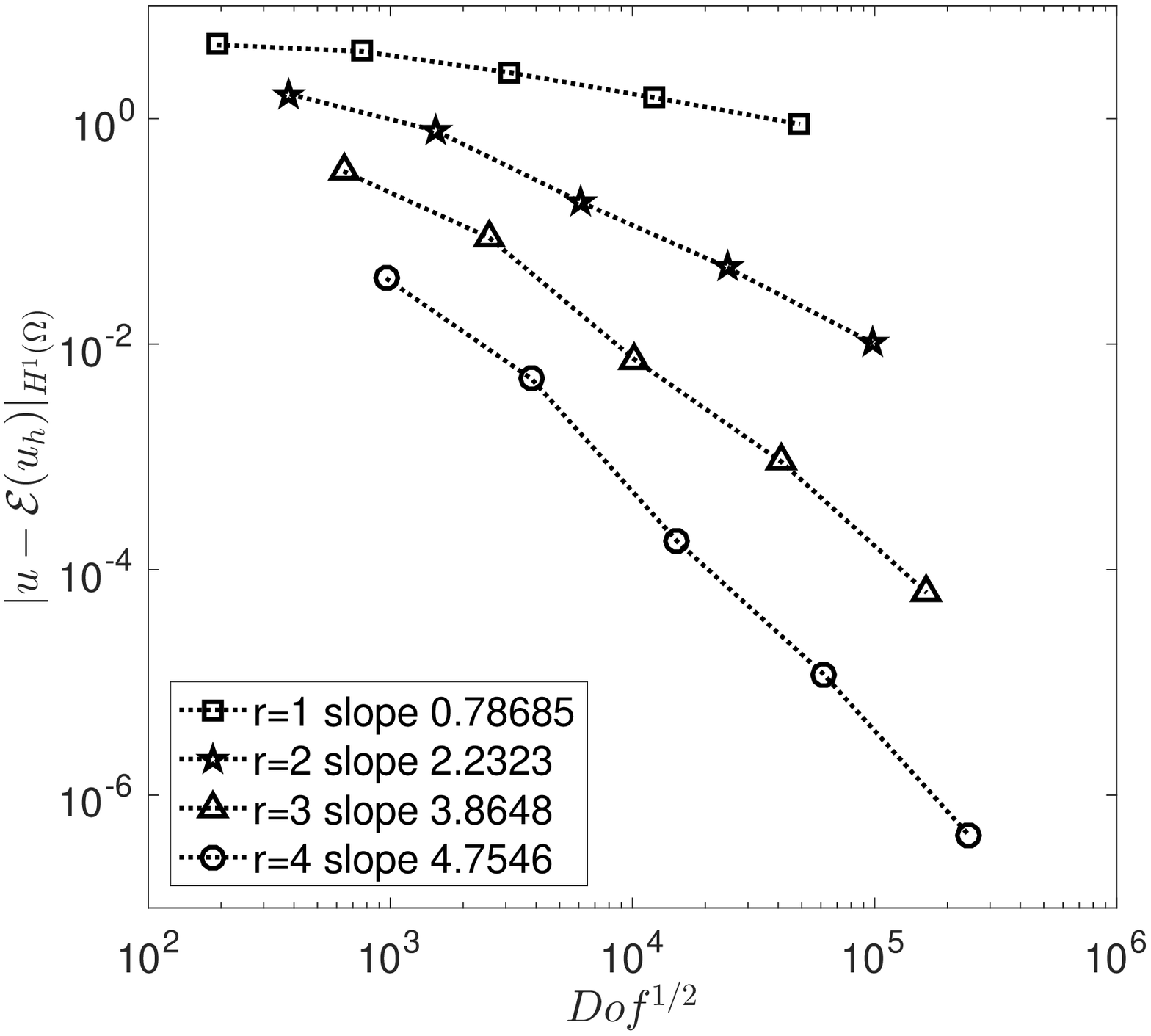}
\caption{
\label{ex4:h-refine}
Example 4. Convergence of the R-FEM under $h$--refinement for $r=1,2,3,4$. } 
\end{figure}

\section*{Acknowledgements} We wish to express our sincere gratitude Andrea Cangiani (University of Leicester) for his insightful comments on an earlier version of this work. ZD and EHG acknowledge funding by The Leverhulme Trust (grant no. RPG-2015-306) and TP acknowledges funding by the EPSRC grant EP/P000835/1.

\bibliographystyle{siam}
\bibliography{RFEMpoly}

\end{document}